\documentclass[12pt,oneside,reqno]{amsart}
\usepackage{amsmath,amsfonts}
\usepackage{esint}

\usepackage{tikz}
\usepackage{dsfont} % pour les indicatrices (\mathds{1})
\usepackage{stmaryrd,mathrsfs,bm,amsthm,mathtools,yfonts,amssymb,color}
\usepackage[a4paper,left=30mm,right=30mm,top=35mm,bottom=35mm,marginpar=25mm]{geometry}

\usepackage{xfrac}
\usepackage{xcolor}
\usetikzlibrary{positioning}
\usepackage{mathtools}

\usepackage[shortlabels]{enumitem}

\usepackage{hyperref} %Leave it as the but one  last package
\usepackage[capitalise, nameinlink, noabbrev]{cleveref} %Leave it as the last package
\hypersetup{colorlinks=true, 
	linkcolor=blue,
	citecolor=red,
}

\theoremstyle{plain} 

\newtheorem{theorem}{Theorem}[section]
\newtheorem{proposition}[theorem]{Proposition}
\newtheorem{cor}[theorem]{Corollary}
\newtheorem{lemma}[theorem]{Lemma}

\newtheorem{assumption}[theorem]{Assumptions}
\theoremstyle{oss}
\newtheorem{remark}[theorem]{Remark}

\newcommand{\R}{\mathbb{R}}

\newcommand{\N}{\mathbb{N}}

\newcommand{\eps}{\varepsilon}

\DeclareMathOperator{\graph}{graph}

\newcommand{\haus}{\mathcal{H}}
\newcommand{\leb}{\mathcal{L}}
\newcommand{\spt}{\mathrm{spt}}

\newcommand{\reg}{\mathrm{reg}}

\newcommand{\sing}{\mathrm{sing}}

\newcommand{\del}{\partial}

\newcommand{\uu}{\underline{u}}
\newcommand{\ou}{\overline{u}}

\makeindex

\def\XXint#1#2#3{{\setbox0=\hbox{$#1{#2#3}{\int}$}
		\vcenter{\hbox{$#2#3$}}\kern-.5\wd0}}

\numberwithin{equation}{section}

\title[Strong maximum principle]{A strong maximum principle for minimizers of the one-phase Bernoulli problem}
\author[N.~Edelen]{Nick Edelen}
\address{\textit{N.~Edelen:} 
Department of Mathematics, University of Notre Dame, Notre Dame, IN, 46556}
\email{nedelen@nd.edu}

\author[L.~Spolaor]{Luca Spolaor}
\address{\textit{L.~Spolaor:} Department of Mathematics, University of California, San Diego, La Jolla, CA, 92093}
\email{lspolaor@ucsd.edu}

\author[B.~Velichkov]{Bozhidar Velichkov}
\address{\textit{B.~Velichkov:} 	
	Dipartimento di Matematica, Universit\`a di Pisa, 
	Largo Bruno Pontecorvo, 5, 56127 Pisa - ITALY}
\email{bozhidar.velichkov@unipi.it}



\begin{document}

	%\tableofcontents
	\begin{abstract} We prove a strong maximum principle for minimizers of the one-phase Alt-Caffarelli functional.  We use this to construct a Hardt-Simon-type foliation associated to any $1$-homogenous global minimizer.		
	\end{abstract}
	\maketitle
	
	\tableofcontents
	
\section{Introduction} 	

In this paper we prove a strong maximum principle for variational solutions of the one-phase Bernoulli problem. For an open set $U\subset \R^d$ and a function $u\in W^{1,2}(U)$, we consider the following functional  introduced by Alt and Caffarelli in \cite{AlCa}
\begin{equation}
J_U(u): = \int_{U}\Big(|D u|^2 + 1_{\{ u > 0 \}}\Big)\,dx. 
\end{equation}
We recall that a nonnegative function $u \in W^{1,2}(U)$ is a {\it minimizer} of $J_U$ (in $U$) if
\[
J_U(u)\leq J_U(u+v)\qquad \text{for every}\qquad v\in W^{1,2}_0(U). 
\]
Similarly, we say that a nonnegative function $u \in W^{1,2}_{loc}(U)$ is a {\it (local-)minimizer} of $J_U$ if  minimizes $J_{U'}$ for all $U' \subset\subset U$; if $U=\R^d$ and $u \in W^{1,2}_{loc}(\R^d)$ is a local-minimizer of $J_{\R^d}$, then we say that $u$ is {\it global minimizer}.

It is well known that if $u\in W^{1,2}_{loc}(U)$ is a minimizer of $J_U$, then it is locally Lipschitz in $U$ and that, denoting with $\Omega_u:=\{u>0\}$ the positivity set of $u$, its free boundary $\del \Omega_u \cap U$ can be decomposed into the disjoint union $\del \Omega_u \cap U=\reg(u)\cup \sing(u)$, where  $\reg(u)$ is relatively open and smooth subset of $\partial\Omega$ and $\sing(u)$ is a closed set of dimension at most $d-5$ (see for instance \cite[Theorems 1.2 and 1.4]{Bojo_ln} and the references therein). Moreover, $u$ solves the overdetermined boundary value problem 
\begin{equation}\label{eqn:EL}
\begin{cases}
\Delta u = 0 &\quad \text{ in } \Omega_u \cap U, \\
u = 0 &\quad \text{ on } \del \Omega_u \cap U\\
D_\nu u = -1 &\quad \text{ on }  \reg(u)\cap U,
\end{cases}
\end{equation}
where $\nu$ denotes the outer unit normal of $\Omega_u \cap U$.

If $u, v \in W^{1,2}(U)$ are minimizers of $J_U$ such that $u\leq v$ (so that $\Omega_u\subset\Omega_v$) and if $\Omega_v$ is connected, then by the classical Hopf maximum principle it follows  that either 
$$u \equiv v\qquad\text{or}\qquad \reg(u)\cap \reg(v)=\emptyset.$$
In this paper we prove a strong ``geometric'' maximum principle, similar to the one known in the minimal surface case (see e.g. \cite{Il, Si_max, SoWh, Wi}), which rules out the \emph{singular} parts of the free boundaries touching either.

\begin{theorem}\label{thm:main}
Let $U\subset\R^d$ be an open set and $u, v \in W^{1,2}_{loc}(U)$ be minimizers of $J_U$.  Suppose that $u \leq v$ and $\reg(u) \cap \reg(v) = \emptyset$ in $U$.  Then $\del\Omega_u \cap \del \Omega_v \cap U = \emptyset$.
\end{theorem}

As an immediate consequence we obtain the following alternative statements of the strong maximum principle. 
 
\begin{cor}\label{cor:main}
Let $U$ be an open set in $\R^d$ and $u, v \in W^{1,2}_{loc}(U)$ be minimizers of $J_U$.  Suppose that $u \leq v$, and $\Omega_v$ is connected.  Then, we have the following dichotomy:
\begin{enumerate}[\quad\rm(i)]
 \item either $u\equiv  v$ in $U$; \smallskip
 \item or $\del\Omega_u \cap \del\Omega_v \cap U = \emptyset$ and $u < v$ on $\Omega_v \supset \overline{\Omega}_u \cap U$.
 \end{enumerate}
\end{cor}

\begin{cor}\label{cor:main2}
Let $U$ be a bounded Lipschitz domain, and let $u, v \in W^{1,2}(U)$ be minimizers of $J_U$.  Suppose that $u \leq v$, and $u < v$ on $\{ x \in \del U : v(x) > 0 \}$.  Then 
$$\del \Omega_u \cap \del \Omega_v \cap U = \emptyset\qquad\text{and}\qquad u < v\quad\text{on}\quad \Omega_v \supset \overline{\Omega}_u \cap U.$$
\end{cor}

We expect \cref{thm:main} to be a useful technical tool, as it has been the case for the analogous result in minimal surface theory. In particular, we demonstrate an application of our strict maximum principle in the following \cref{thm:fol}, which proves the existence of a ``Hardt-Simon''-type foliation associated to \emph{any} $1$-homogenous minimizer, again analogous to the one known for area-minimizing hypercones (see e.g. \cite{BoDeGi, HaSi, Wa}).  We mention that \cite{DeJeSh} contains versions of \cref{thm:main}, \cref{thm:fol} for minimizers with isolated singularities (see also \cref{rem:fol}); our maximum principle, for general minimizers, is proven using a fundamentally different approach, and the increased generality is the reason we are able to prove existence (but not uniqueness!) of the foliation in greater generality also.
\begin{theorem}\label{thm:fol}
Let $u_0 \in W^{1,2}_{loc}(\R^d)$ be a global $1$-homogeneous minimizer of $J_{\R^d}$.  Then there exist global minimizers $\underline{u}, \overline{u} \in W^{1,2}_{loc}(\R^d)$ such that
\begin{enumerate}[\rm(1)]
\item $\underline{u} \leq u_0 \leq \overline{u}$; \label{item:fol-1}\smallskip
\item $d(0, \Omega_{\underline{u}}) = d(0, \Omega_{\overline{u}}) = 1$; \label{item:fol-2}\smallskip
\item $-\uu(x) + x \cdot D\uu(x) > 0$ for $x \in \overline{\Omega}_{\underline{u}}$\,, and $-\ou(x) + x \cdot D\ou(x) < 0$ for $x \in \overline{\Omega}_{\ou}$\,; \label{item:fol-3} \smallskip
\item $\sing(\underline{u}) = \sing(\overline{u}) = \emptyset$; \label{item:fol-4} \smallskip
\item $\underline{u}_{0, r}\to u_0$ and $\overline{u}_{0, r} \to u_0$ in $(W^{1,2}_{loc} \cap C^\alpha_{loc})(\R^d)$ as $r \to \infty$. \label{item:fol-5}
\end{enumerate}
In particular, the hypersurface $\del \Omega_{\uu}$ (resp. $\del\Omega_{\ou}$) is an analytic radial graph over $\Omega_{u_0} \cap \del B_1$ (resp. $\del B_1 \setminus \overline{\Omega}_{u_0}$), and the dilations
\[
\big\{\lambda \del \Omega_{\underline{u}}\ :\ \lambda>0\big\}\cup \big\{\lambda \del \Omega_{\overline{u}}\ :\ \lambda>0\big\} 
\]
foliate $\R^d \setminus \del \Omega_{u_0}$.
\end{theorem}

\begin{remark}\label{rem:fol}
Note that, unlike the case when $u_0$ is regular away from $0$ as considered in \cite{DeJeSh}, we do \emph{not} claim any uniqueness of the foliation generated by $\underline{u}, \overline{u}$.  We expect the foliation should be unique, like in \cite{DeJeSh}, in the sense that any minimizer lying to one side of $u_0$ should be a dilation of either $\underline{u}$ or $\overline{u}$, but this seems to be a much more subtle question.
\end{remark}

\subsection*{Outline of the proof and organization of the paper} The key technical tools in the proof of \cref{thm:main} are two relative isoperimetric inequalities (\cref{ss:2}), which allow us to deduce  Gagliardo-Nirenberg-Sobolev-type inequalities (\cref{ss:3}) and to develop a De\,Giorgi-Nash-Moser theory (\cref{ss:4}) for sub and supersolutions on domains $\Omega_u$, generated by minimizers  $u$ of the one-phase functional $J$.  To prove these we use ideas from \cite{BoGi} and \cite{Si_dec}.\medskip

Beyond the Harnack inequalities, our strategy of proof for \cref{thm:main} essentially follows the method of \cite{Si_max} (see \cref{ss:6}).  We assume that $u \neq v$ and $\reg(u) \cap \reg(v) = \emptyset$ but $\del \Omega_u \cap \del \Omega_v \cap U \neq \emptyset$, and derive a contradiction.  We first show using a dimension reduction argument that there is no loss in assuming that $U = B_1$ and $0 \in \del \Omega_u \cap \del\Omega_v$, and both $u$, $v$ have the same tangent cone at $0$ (for \emph{any} choice of rescalings).  This implies that the difference $u - v$ behaves like $o(r)$, and so by choosing a good sequence $r_i \to 0$ and suitable factors $\lambda_i$, we can find a blow-up $u_0$ of both $u$ and $v$ at $0$, and can take a limit of $\lambda_i^{-1}(v_{0, r_i} - u_{0, r_i})$ to obtain a positive Jacobi field $w$ on $\{ u_0 > 0\} \cap B_1$ which behaves like $O(r)$ as $r\to0$.  However, as $w$ is a positive (distributional) supersolution of the Neumann Laplacian (see \cref{ss:5}, \cref{ss:6}), that is, 
$$\Delta w\le 0\quad\text{and}\quad w\ge 0\quad\text{on}\quad \{ u_0 > 0\} \cap B_1,$$ the De\,Giorgi-Nash-Moser Harnack inequality implies that $w$ admits a uniform lower bound, contradicting the fact $w = O(r)$.

\medskip

In \cref{ss:1} we recall some useful facts about minimizers of the one-phase Bernoulli energy $J$. In \cref{ss:2} we prove a relative isoperimetric inequality and a relative Neumann-type isoperimetric inequality for compact domains in $\Omega_u$, $u$ a minimizer of $J$, and then use these in \cref{ss:3} to prove a Sobolev and Neumann-Sobolev inequality. \cref{ss:4} summarizes how these Sobolev inequalities imply the De\,Giorgi-Nash-Moser estimates. In \cref{ss:5} we show how sequences $u^\mu < v^\mu$ of minimizers to $J$ can be rescaled to obtain a Jacobi field on the limit, largely following work of \cite{DeJeSh}. Finally in \cref{ss:6}, \cref{ss:fol} we combine the results of the previous two sections to prove \cref{thm:main}, \cref{thm:fol}.

 \medskip
 
 \noindent\textbf{Acknowledgements.} N.E. thanks Stanford University and UC San Diego for their hospitality. L.S. has been partially supported by the N.S.F Career Grant DMS 2044954.  B.V. was supported by the European Research Council (ERC) and the European Union's programme Horizon 2020 through the project ERC VAREG - {\it Variational approach to the regularity of the free boundaries} \rm (No. 853404).
 
 \section{Preliminary results}\label{ss:1} 
 
 In this section we recall some facts about minimizers of the one-phase energy $J_U$.  Given a minimizer $u$ of $J_U$, we shall always write $\Omega_u = \{ u > 0\}$ for the positive set, and $u_{x, r}(y) := r^{-1} u(x + ry)$ for the scaled/translated function.  For a general function $f$ we write $f^+ = \max \{ f, 0\}$, and $f^- = - \min\{ f, 0 \}$.  For a set $A \subset \R^d$, write $d(x, A)$ for the Euclidean distance from $x$ to $A$.

 We start by recalling the standard compactness for minimizers of the one-phase problem. 
 
\begin{lemma}[Compactness of minimizers]\label{lem:comp}
Let $\{ u_i \in W^{1,2}_{loc}(B_1) \}_i$ be a sequence of minimizers of $J_{B_1}$, and suppose that $0 \in \del\Omega_{u_i}$ for all $i$.  Then after passing to a subsequence, we can find a $u \in W^{1,2}_{loc}(B_1)$ such that:
\begin{enumerate}[\quad\rm(1)]
\item $u_i \to u$ in $(C^\alpha_{loc} \cap W^{1,2}_{loc})(B_1)$ for all $\alpha < 1$;
\item the characteristic functions $1_{\Omega_{u_i}} \to 1_{\Omega_u}$ in $L^1_{loc}(B_1)$;
\item the free-boundaries $\del\Omega_{u_i} \to \del \Omega_u$ in the local Hausdorff distance in $B_1$;
\item $u$ minimizes $J_{B_1}$.
\end{enumerate}
\end{lemma}

\begin{proof}
This is proven in \cite[Lemmas 3.2, 3.4, and Section 4.7]{AlCa}.
\end{proof}

In order to prove the desired isoperimetric inequalities in \cref{ss:2}, we will also need the following density bounds.
 
 \begin{lemma}[Density bounds]\label{lem:ahlfors}
There exists a dimensional constant $\beta=\beta(d) > 0$ so that if $u \in W^{1,2}(B_2)$ minimizes $J_{B_2}$, $0 \in \overline{\Omega_u}$, then
\begin{equation}\label{eqn:ahlfors-concl1}
\haus^{d-1}(\del\Omega_u \cap B_1) \leq \omega_{d-1} \beta^{d-1},
\end{equation}
and if $\Omega'$ is any connected component of $\Omega_u \cap B_2$ satisfying $0 \in \overline{\Omega'}$, then
\begin{equation}\label{eqn:ahlfors-concl2}
\haus^{d}(\Omega' \cap B_1) \geq \frac{\omega_d}{\beta^d}.
\end{equation}
In fact, we can find a ball $B_{\beta^{-1}}(y) \subset \Omega' \cap B_1$ in which $u \geq 1/\beta$.
\end{lemma}

\begin{proof}
The upper bound \eqref{eqn:ahlfors-concl1} follows from \cite[Theorem 4.5(3)]{AlCa} (or \cite[Corollary 5.8]{Bojo_ln}).  The lower bound \eqref{eqn:ahlfors-concl2} follows from the Lipschitz nature of $u$ and a minor modification of \cite[Lemma 3.4]{AlCa} (or \cite[Lemma 5.1(d)]{Bojo_ln}).  Specifically, observe that if $v \in W^{1,2}(B_2)$ satisfies $v|_{\del B_2} = u|_{\del B_2}$, then the function
\[
v'(x) = \left\{ \begin{array}{l l} u(x) & x \not\in \Omega' \\ \min\{ u(x), v(x)\} & x \in \Omega' \end{array} \right.
\]
also lies in $W^{1,2}(B_2)$ and agrees with $u$ on $\del B_2$.  Therefore we have the inequality
\[
J_{\Omega'}(v') \leq J_{\Omega'}(u).
\]
Since we also have $u \cdot 1_{\Omega'} \in W^{1,2}(B_2)$, we can therefore apply the same proof of \cite[Lemma 3.4]{AlCa} to $u|_{\Omega'}$ in place of $u$ to deduce
\[
\sup_{\Omega' \cap B_{1/2}} u \geq 1/c(n).
\]
Since (by \cite[Corollary 3.3]{AlCa}) we also have $||Du||_{L^\infty(B_{1})} \leq c(n)$, it follows that we can find a $y \in \Omega' \cap B_{1/2}$ and a $\beta(n) \geq 4$ so that $u \geq 1/\beta$ on $B_{1/\beta}(y)$, which concludes the proof of the lower bound \eqref{eqn:ahlfors-concl2}.
%  This proves the Lemma.
\end{proof}

A general minimizer $u$ on some bounded open domain $U$ might have numerous connected components of $\Omega_u$.  However if $u$ is a $1$-homogenous and $U=\R^d$, then $\Omega_u$ must be connected, essentially due to the fact that any eigenfunction on the sphere $S^{d-1}$ with eigenvalue $(d-1)$ must be the restriction of a linear function.  This implies the following connectivity result for global minimizers, which is analogous to \cite[Theorem 1]{BoGi}.
\begin{theorem}\label{lem:conn}
Let $u \in W^{1,2}_{loc}(\R^d)$ be a global minimizer of $J_{\R^d}$.  Then $\Omega_u$ is connected.
\end{theorem}

\begin{remark}
The same proof (taking $r_k \to 0$ instead of $\to \infty$) implies that if $u \in W^{1,2}(B_1)$ minimizes $J_{B_1}$, then for any $p \in B_1$ there is at most one connected component of $\Omega_u$ whose closure contains $p$.
\end{remark}

\begin{proof}
We first prove the Theorem for $u$ being $1$-homogenous.  In this case the argument is similar to \cite[Lemma 2.2]{DeSpVe}.  Indeed, suppose by contradiction $\Omega_{u}$ has two non-empty disjoint connected components $\Omega_1, \Omega_2$.  Since $u$ is $1$-homogenous and solves $\Delta u = 0$ in $\Omega_{u}$, we can write $u(r \theta) = r z(\theta)$, where $z \in W^{1,2}_0(\Omega_{u} \cap \del B_1)$ solves
\begin{equation}\label{eqn:conn-1}
\Delta_{S^{d-1}} z + (d-1) z = 0 \text{ on } \Omega_{u} \cap \del B_1.
\end{equation}
Write $z_i = z|_{\Omega_i}$, so that each $z_i$ is a non-negative Dirichlet eigenfunction of the spherical Laplacian $\Delta_{S^{d-1}}$ on $\Omega_i \cap \del B_1$ with eigenvalue $d-1$.

Choose $a > 0$ so that
\[
\int_{\del B_1} (z_1 - a z_2) d\haus^{d-1} = 0,
\]
and then observe that by \eqref{eqn:conn-1} and an integration by parts we have
\[
\int_{\del B_1} |D_{\theta}(z_1 - a z_2)|^2 d\haus^{d-1} = (d-1) \int_{\del B_1} |z_1 - az_2|^2 d\haus^{d-1}.
\]
That is, $z_1 - az_2$ is a first (non-trivial) eigenfunction of $\del B_1$, and hence must be the restriction to $\del B_1$ of a linear function.  After a rotation, we deduce $u$ must take the form 
\[
u = \alpha x_d^+ + \beta x_d^-
\]
for some $\alpha, \beta > 0$.  But now $\haus^d(\Omega_u) = 0$, and $u$ is not itself harmonic, and so if $v$ is the harmonic extension of $u|_{\del B_1}$ to $B_1$ we have $J_{B_1}(v) < J_{B_1}(u)$, contradicting minimality of $u$.  This proves \cref{lem:conn} when $u$ is $1$-homogenous.

Now take a general $u$ as in the statement of the Theorem, and suppose, towards a contradiction, there are two disjoint, non-empty connected components $\Omega_1, \Omega_2 \subset \Omega_u$.  Pick any sequence $r_k \to \infty$.  For $k >> 1$ and $i = 1, 2$, $r_k^{-1}\Omega_i \cap B_{1/100} \neq \emptyset$, and therefore by \cref{lem:ahlfors} we can find balls $B_{1/\beta}(y_{ik}) \subset r_k^{-1} \Omega_i \cap B_2$ on which $u \geq 1/\beta$.

Passing to a subsequence, by standard compactness (\cref{lem:comp}) and the Weiss monotonicity formula, we can assume there is a $1$-homogeneous $u_0 \in W^{1,2}_{loc}(\R^d)$, minimizing $J_{\R^d}$, so that $u_{0, r_k} \to u_0$ in $C^\alpha_{loc}$.  By our choice of $y_{ik}$ and the $C^0_{loc}$ converge of the $u_{0, r_k}$, after passing to a further subsequence can additionally assume that $y_{ik} \to y_i \in \Omega_{u_0} \cap B_2$, for each $i = 1, 2$.

By Step 1 there is a path $\gamma : [0, 1] \to \Omega_{u_0} \cap B_2$ connecting $y_1$ to $y_2$.  By the $C^0_{loc}$ convergence of the $u_{0, r_k}$, we deduce that $\gamma([0, 1]) \subset r_k^{-1} \Omega_u$ for $k >> 1$.  Provided $k >> 1$ so that, additionally, each $y_i \in B_{1/\beta}(y_{ik})$, we deduce there is a path in $r_k^{-1} \Omega_u$ connecting $y_{1k}$ to $y_{2k}$.  This is a contradiction, and finishes the proof of \cref{lem:conn}.
\end{proof}

We will also need the following property of global minimizers.

 \begin{lemma}\label{lem:H}
 	Let $u \in W^{1,2}_{loc}(\R^d)$ be a global minimizer for $J_{\R^d}$. Then $\sup_{\Omega_u} |Du| = 1$.  As a consequence, if $H$ is the mean scalar curvature of $\reg(u)$ with respect to the outer unit normal, then $H \leq 0$ (and $H < 0$ if $u$ is not linear).
 \end{lemma}
 
 \begin{proof}
 	Define
 	\[
 	\Lambda = \sup \left\{ \sup_{\Omega_u} |Du| : \text{ $u \in W^{1,2}_{loc}(\R^d)$ a global minimizer of $J_{\R^d}$} \right\},
 	\]
 	and notice that, since $|Du| = 1$ on $\reg(u)$, we have that $\Lambda \geq 1$.
 	
 	Suppose, towards a contradiction, that $\Lambda > 1$.  Then, there is a sequence of global minimizers $u_i \in W^{1,2}_{loc}(\R^d)$ and points $x_i \in \Omega_{u_i}$ so that $|Du_i(x_i)| \to \Lambda$.  Let $y_i \in \del\Omega_{u_i}$ realize $d(x_i, \del\Omega_{u_i})$.  After a translation/rotation/dilation, since $|Du|$ is scale-invariant, we can assume $x_i = e_d$ and $y_i = 0$.
 	
 	Passing to a subsequence, by \cref{lem:comp} we can assume there is a $u \in W^{1,2}_{loc}(\R^d)$ minimizing $J_{\R^d}$ so that $u_i \to u$ in $(C^\alpha_{loc} \cap W^{1,2}_{loc})(\R^d)$, and $\del \Omega_{u_i} \to \del \Omega_{u}$ in the local Hausdorff distance, and $u_i \to u$ in $C^\infty_{loc}(\Omega_u)$.  Since $d(e_d, \del \Omega_{u_i}) = 1$, we have $d(e_d, \del \Omega_{u}) = 1$.  So $e_d \in \Omega_{u}$ and $|Du(e_d)| = \Lambda$.  (Note this implies $\Lambda < \infty$).  On the other hand, $|Du| \leq \Lambda$.  Therefore $e_d$ is an interior maximum for $|Du|^2$.
 	
 	Since $\Delta |Du|^2 \geq 0$, $|Du|^2$ must be locally constant, and hence $u = x_d^+$.  This implies $|Du(e_d)| = 1 < \Lambda$, which is a contradiction and concludes the proof of the first claim of the lemma. We are now in position to prove the second assertion of the Lemma. By the previous one, we have that  
 	$$\Delta |Du|^2 \geq 0\quad\text{and}\quad |Du| \leq 1\quad\text{in}\quad \Omega_u\,.$$
 	On the other hand, on the regular part of the free boundary, we have: 
 	 	$$|Du| = 1\quad\text{and}\quad D_\nu |Du|^2 = -H\quad\text{on}\quad \reg(u)\,,$$
so the conclusion follows from the Hopf lemma.
 \end{proof}

Finally we recall the following $\eps$-regularity theorem due to Alt-Caffarelli \cite{AlCa}, which we state in the version of De Silva \cite{desilva}.

\begin{theorem}[Alt-Caffarelli $\eps$-regularity]\label{thm:ac-reg}
Given $\eps > 0$, there is a $\delta>0$, depending on $\eps, d$, such that if $u \in W^{1,2}(B_1)$ is a minimizer of $J_{B_1}$ and 
\begin{equation}\label{eqn:flat}
||u - x_d^+||_{L^\infty(B_1)} < \delta\,,
\end{equation} then $u \in C^\infty( B_{1-\eps} \cap \overline{\{u > 0\}})$ and there is a $C^\infty$ function 
$$\xi : B_{1-\eps}\cap \{x_d=0\}\to\R$$ 
such that 
\begin{gather}
\del \Omega_u \cap B_{1-\eps}=\graph(\xi)\cap B_{1-\eps}\,,\quad \text{with}\quad \|\xi\|_{C^{3,1}(B_{1-\eps}\cap \{x_d=0\})} \leq \eps\,, \label{eqn:ac1}\\
\|u\|_{C^{3,1}(\overline{\Omega_u} \cap B_{1-\eps})} \leq C(d)\,, \label{eqn:ac2}\\
\|Du - e_d\|_{L^\infty(B_{1-\eps} \cap \{ u > 0 \})} \leq \eps\,.\label{eqn:ac3}
\end{gather}
\end{theorem}
 
 \begin{proof} This theorem with $C^{1,\alpha}$ norms replacing $C^{3,1}$ was proved by De Silva in \cite{desilva}. The higher order regularity is a standard consequence of \cite[Theorem 2]{KiNi}. 
 \end{proof}

 \section{Isoperimetric inequalities}\label{ss:2}
   
 In this section we prove two types of isoperimetric inequalities for domains $\Omega_u$, with $u$ a minimizer of $J$.
 
 \subsection{Relative isoperimetric inequality}  The proof of the following theorem follows ideas from \cite{Si_dec}.
   
 \begin{theorem}[Relative isoperimetric inequality]\label{thm:iso}
There are dimensional constants $R_1>0$ and $C_1>0$ so that if $u \in W^{1,2}(B_{R_1})$ is a minimizer for $J_{B_{R_1}}$, then
\[
\haus^d(Q \cap \Omega_u)^{(d-1)/d} \leq C_1(d) \haus^{d-1}(\del Q \cap \Omega_u),
\]
for any set $Q \subset \Omega_u \cap B_1$, with $\del Q \cap \Omega_u$ being countably $(d-1)$-rectifiable.
\end{theorem}

\begin{proof}  Let $\beta=\beta(d)>0$ be as in \cref{lem:ahlfors} and define 
\begin{equation}\label{eqn:choice_const}
\theta = \frac12 \min\{ 2^{-d} \beta^{-d}, 1\} \qquad\text{ and }\qquad R = \max\{ 4 (\theta/2)^{-1/d}, 8\}\,.
\end{equation}
  Suppose, towards a contradiction, \cref{thm:iso} failed.  Then there is a sequence $u_k \in W^{1,2}(B_R)$ minimizing $J_{B_R}$, and a sequence $Q_k$ of compact subsets of $\overline{\Omega_k \cap B_1}$, for $\Omega_k := \Omega_{u_k}$, with $\del Q_k \cap \Omega_k$ rectifiable, such that 
\begin{equation}\label{eqn:iso-1}
\haus^d(Q_k \cap \Omega_k)^{(d-1)/d} \geq k \haus^{d-1}(\del Q_k \cap \Omega_k).
\end{equation}
Notice that 
\[
\lim_{r \to 0} \frac{\haus^d(Q_k \cap \overline{B_r(x)})}{\omega_d r^d} = 1 > \theta\,,\quad \haus^d-\text{a.e. } x \in Q_k\,.
\]
On the other hand, since $Q_k \subset B_1$ and recalling our choice of $R$,
\[
\frac{\haus^d(Q_k \cap \overline{B_{R/4}(x)})}{\omega_d (R/4)^d} < \theta\,,\qquad \forall x\in Q_k.
\]
Therefore, there is a subset $\tilde Q_k \subset Q_k$ with $\haus^d(\tilde Q_k \setminus Q_k) = 0$, so that for every $x \in \tilde Q_k$ we can find an $r_x\in(0,R/4)$ satisfying
\[
\inf_{r < r_x} \frac{\haus^d(Q_k \cap \overline{B_r(x)})}{\omega_d r^d} = \frac{\haus^d(Q_k \cap \overline{B_{r_x}(x)})}{\omega_d r_x^d} = \theta.
\]
\noindent Fix momentarily a $k$.  By the Besicovich covering theorem, we can find a subcollection $\{ B_{r_i}(x_i)\}_i \subset \{ B_{r_x}(x) : x \in \tilde Q_k \}$ so that $\tilde Q_k \subset \cup_i \overline{B_{r_i}(x_i)}$ and the balls $\{ \overline{B_{r_i}(x_i)}\}_i$ divide into at most $N(d)$ disjoint subfamilies.  We claim that if $k >> 1$, then for at least one $i$ we must have
\begin{equation}\label{eqn:cont1}
\haus^d(Q_k \cap \overline{B_{r_i}(x_i)})^{(d-1)/d} \geq \sqrt{k} \haus^{d-1}(\del Q_k \cap \Omega_k \cap \overline{B_{r_i}(x_i)}).
\end{equation}
Otherwise, we could estimate
\begin{align*}
\haus^d(Q_k)^{(d-1)/d}
&\leq \bigg( \sum_i \haus^d(Q_k \cap \overline{B_{r_i}(x_i)}) \bigg)^{(d-1)/d} \\
&\leq \sum_i \haus^d(Q_k \cap \overline{B_{r_i}(x_i)})^{(d-1)/d} \\
&\stackrel{\eqref{eqn:cont1}}{\leq} \sqrt{k} \sum_i \haus^{d-1}(\del Q_k \cap \Omega_k \cap  \overline{B_{r_i}(x_i)}) \\
&\leq \sqrt{k} \,N(d)\, \haus^{d}(\del Q_k \cap \Omega_k),
\end{align*}
which contradicts \eqref{eqn:iso-1}, if $k$ is chosen sufficiently large, depending on the dimension.

After translating and homogeneously rescaling $u_k$, $\Omega_k$, $Q_k$, and considering only $k$ sufficiently large, we can therefore assume that $u_k \in W^{1,2}(B_2)$ is a minimizer of $J_{B_2}$, with $0 \in \overline{\Omega_k}$ and
\begin{equation}\label{eqn:iso-2}
\haus^d(Q_k \cap \overline{B_1})^{(d-1)/d} \geq \sqrt{k} \haus^{d-1}(\del Q_k \cap \Omega_k \cap \overline{B_1}) , 
\end{equation}
and
\begin{equation}\label{eqn:meas}
\inf_{r < 1} \frac{\haus^d(Q_k \cap \overline{B_r})}{\omega_d r^d} = \frac{\haus^d(Q_k \cap \overline{B_1})}{\omega_d} = \theta.
\end{equation}
Passing to a subsequence, we can assume that for all $k$ we have either $B_{3/2} \subset \Omega_k$ or $B_{3/2} \not\subset \Omega_k$.  Suppose the latter occurs.  By \cref{lem:comp}, there is a minimizer $u \in W^{1,2}_{loc}(B_2)$ of $J_{B_2}$, so that up to subsequences $u_k \to u$ in $C^\alpha_{loc}(B_2) \cap W^{1,2}_{loc}(B_2)$, $1_{\Omega_k} \to 1_{\Omega}$ in $L^1_{loc}(B_2)$ and the free boundaries converge in the local Hausdorff distance in $B_2$, where $\Omega := \Omega_u$ (and is such that $0 \in \overline{\Omega}$).

Notice that $\del Q_k = (\del Q_k \cap \Omega_k) \cup (\overline{Q_k} \cap \del\Omega_k)$ is closed, $(d-1)$-rectifiable, with finite $(d-1)$-Hausdorff measure, so that using \eqref{eqn:ahlfors-concl1} and \eqref{eqn:iso-2} we deduce that each $Q_k$ is a set of finite perimeter in $B_1$, with
\begin{equation}\label{eqn:iso-4}
\haus^{d-1}( \del^* Q_k \cap B_1) \leq \haus^{d-1}(\del Q_k \cap \Omega_k \cap B_1) + \haus^{d-1}(\del \Omega_k \cap B_1) \leq C(d)
\end{equation}
Therefore the compactness theory for sets of locally-finite perimeter implies there is a set $Q \subset \Omega \cap B_1$ so that (after passing to a further subsequence) $1_{Q_k} \to 1_Q$ in $L^1_{loc}(B_1)$.  From \eqref{eqn:iso-2}, and the local Hausdorff convergence $\del\Omega_k \to \del\Omega$, we have
\[
\int_Q {\rm div}(\phi) \,dx= 0\quad\text{for every}\quad \phi \in C^1_c(\Omega \cap B_1; \R^d)\,.
\]
Therefore $1_Q$ is locally-constant on $\Omega \cap B_1$, and hence $Q$ is a union of connected components of $\Omega \cap B_1$.

From \eqref{eqn:ahlfors-concl2}, there are only finitely-many connected components of $\Omega_u$ meeting $B_{1/100}$.  Since by \eqref{eqn:meas} we have $\haus^d(Q \cap B_r) \geq \theta \omega_d r^d > 0$ for all $r < 1$, we deduce $Q$ must contain a connected component $\Omega'$ of $\Omega \cap B_1$ such that $0 \in \overline{\Omega'}$.  Applying again \eqref{eqn:ahlfors-concl2} we deduce that
\begin{equation}\label{eqn:iso-3.1}
\haus^d(Q \cap B_{1/2}) \geq \haus^d(\Omega' \cap B_{1/2}) \geq \frac{\omega_d}{2^d \beta^d}.
\end{equation}
On the other hand, from \eqref{eqn:meas} we have
\begin{equation}\label{eqn:iso-3.2}
\haus^d(Q\cap B_1) \leq \omega_d \theta ,
\end{equation}
which by our choice of $\theta$ contradicts \eqref{eqn:iso-3.1}.

Finally, we notice that if $B_{3/2} \subset \Omega_k$ for all $k$, then in the above discussion we can simply replace $\Omega_k$ and $\Omega$ with $B_{3/2}$, and deduce the same contradiction.  This concludes the proof of  \cref{thm:iso}.
\end{proof}

\subsection{Neumann-isoperimetric} In this subsection we follow \cite{BoGi}.

\begin{theorem}[Neumann-type isoperimetric inequality]\label{thm:niso}
There is a positive constant $\gamma=\gamma(d)$ so that if $u \in W^{1,2}(B_{R_1})$ minimizes $J_{B_{R_1}}$, with $R_1>0$ as in \cref{thm:iso}, then
\[
\min\left\{ \haus^d(B_\gamma \cap Q), \haus^d(B_\gamma \cap \Omega_u \setminus Q) \right\}^{(d-1)/d} \leq \gamma^{-1} \haus^{d-1}(\del Q \cap \Omega_u \cap B_1)
\]
for all $Q \subset \Omega_u \cap B_1$ such that $\del Q \cap \Omega_u \cap B_1$ is countably $(d-1)$-rectifiable.
\end{theorem}

\begin{proof}
Suppose \cref{thm:niso} failed.  Then we could find a sequence $u_k \in W^{1,2}(B_{k R_1})$ of minimizers of $J_{B_{R_1}}$ and a sequence $Q_k \subset \Omega_k \cap B_k$, where $\Omega_k := \Omega_{u_k}$, so that
\begin{equation}\label{eqn:niso-1}
\min\{ \haus^d(B_{1/k} \cap Q_k), \haus^d(B_{1/k} \cap \Omega_k \setminus Q_k) \}^{(d-1)/d} \geq k \haus^{d-1}(\del Q_k \cap \Omega_k \cap B_k),
\end{equation}
and $\del Q_k \cap \Omega_k \cap B_k$ is $(d-1)$-rectifiable.  Let us write $Q_k' = B_{k} \cap \Omega_k \setminus Q_k$.

Notice that \eqref{eqn:niso-1} implies
\begin{equation}\label{eqn:niso-2}
\haus^{d-1}(\del Q_k \cap \Omega_k \cap B_t) \leq \frac{1}{k} \haus^d(Q_k \cap B_t )^{(d-1)/d}
\end{equation}
for all $1/k \leq t \leq k$.  Then from the isoperimetric inequality of \cref{thm:niso}, and the coarea formula, we estimate for a.e. $1/k \leq t \leq k$:
\begin{align*}
\haus^d(Q_k \cap B_t)^{(d-1)/d}
&\leq C_1 \haus^{d-1}(\del(Q_k \cap B_t) \cap \Omega_k) \\
&\leq C_1 \haus^{d-1}(\del Q_k \cap \Omega_k \cap B_t) + C_1 \haus^{d-1} (Q_k \cap \del B_t) \\
&\stackrel{\eqref{eqn:niso-2}}{\leq} \frac{C_1}k \haus^d(Q \cap B_t)^{(d-1)/d} + C_1 \frac{d}{dt} \haus^d(Q_k \cap B_t) ,
\end{align*}
and therefore, for sufficiently large $k$, we can estimate
\begin{equation}\label{eqn:niso-3}
\haus^d(Q_k \cap B_t) \geq \frac{1}{2C_1 d} \left(t - \frac1k\right)^d\qquad \text{for all $1/k \leq t \leq k$.}
\end{equation}
  Since \eqref{eqn:niso-1} implies that \eqref{eqn:niso-2} holds with $Q_k'$ in place of $Q_k$, with the same reasoning as above we have
 \begin{equation}\label{eqn:niso-31}
\haus^d(Q_k' \cap B_t) \geq \frac{1}{2C_1 d} \left(t - \frac1k\right)^d\qquad \text{for all $1/k \leq t \leq k$.}
\end{equation}
Note \eqref{eqn:niso-3} implies $\Omega_k \cap B_1 \neq \emptyset$ for all $k >> 1$.

After passing to a subsequence, we can assume that 
$$\text{either }\quad d(0, \del\Omega_k) \to \infty\quad \text{  or }\quad \sup_k d(0, \del\Omega_k) < \infty\,.
$$
  Suppose the latter occurs.  Passing to a further subsequence, by \cref{lem:comp} we can assume there is a minimizer $u \in W^{1,2}_{loc}(\R^d)$, so that $u_k \to u$ in $C^\alpha_{loc} \cap W^{1,2}_{loc}$.  Write $\Omega = \Omega_u$, then $\del\Omega_k \to \del \Omega$ in the local Hausdorff distance and $1_{\Omega_k} \to 1_\Omega$ in $L^1_{loc}$.

Arguing as in the proof of \cref{thm:iso}, from \eqref{eqn:niso-1}, \cref{lem:ahlfors} and the compactness theory for sets of locally-finite perimeter (passing to a yet further subsequence), we can assume there are sets of locally-finite perimeter $Q, Q' \subset \Omega$ so that
\[
1_{Q_k} \to 1_Q, \quad 1_{Q_k'} \to 1_Q \quad \text{ in } L^1_{loc}.
\]
From  \eqref{eqn:niso-3} and \eqref{eqn:niso-31}, we have
\[
\haus^d(Q \cap B_t) \geq t^d/c(d), \quad \haus^d(Q' \cap B_t) \geq t^d/c(d) \quad \forall t > 0,
\]
while from \eqref{eqn:niso-2}, each $1_Q, 1_{Q'}$ is locally-constant in $\Omega$.  Since by \cref{lem:conn}, $\Omega$ is connected, we deduce that $Q = Q' = \Omega$ up to a set of $\haus^d$-measure zero.  However, since every $Q_k \cap Q_k' = \emptyset$, we have $Q \cap Q' = \emptyset$ up to a set of $\haus^d$-measure zero.  This is a contradiction.

Suppose $d(0, \del\Omega_k) \to \infty$.  Then from \eqref{eqn:niso-3} we can find a sequence $t_k \to \infty$ so that $B_{t_k} \subset \Omega_k$.  In the above discussion we can replace $\Omega_k$ with $B_{t_k}$ and $\Omega$ with $\R^d$ to deduce a contradiction as before.  This proves \cref{thm:niso}.
\end{proof}

 \section{Sobolev inequalities}\label{ss:3}

 The isoperimetric inequalities of \cref{ss:2} imply a Sobolev and a Neumann-Sobolev inequality.

\begin{theorem}\label{thm:sobolev}
There are dimensional constants $R, C \geq 1$, and $\gamma \in (0, 1]$ so that if $u \in W^{1,2}(B_R)$ minimizes $J_{B_R}$ and $f \in W^{1,1}(\Omega_u \cap B_1)$, then
\begin{equation}\label{eqn:sobolev-concl1}
\inf_k \bigg( \int_{\Omega_u \cap B_\gamma} |f - k|^{d/(d-1)} \bigg)^{(d-1)/d} \leq C \int_{\Omega_u \cap B_1} |Df|.
\end{equation}
If $\spt f \subset B_1$, then 
\begin{equation}\label{eqn:sobolev-concl2}
\left( \int_{\Omega_u \cap B_1} |f|^{d/(d-1)} \right)^{(d-1)/d} \leq C \int_{\Omega_u \cap B_1} |Df|.
\end{equation}
\end{theorem}

By a standard application of Holder's inequality (see for instance \cite[Section 5.6.1, Theorem 1]{Evans}), we have:
\begin{cor}\label{cor:sobolev}
In the notation of \cref{thm:sobolev}, if $f \in W^{1,2}(\Omega_u \cap B_1)$ is supported in $B_1$, then
\[
\left( \int_{\Omega_u \cap B_1} |f|^{2\chi} \right)^{1/\chi} \leq C(d, \chi) \int_{\Omega_u \cap B_1} |Df|^2,
\]
where $\chi = d/(d-2)$ if $d \geq 3$, or $\chi > 1$ is arbitrary if $d = 2$.
\end{cor}

\noindent In the proof of \cref{thm:sobolev} we will make use of the following well-known inequality. 
\begin{lemma}[Hardy-Littlewood-Polya]
Let $V:[0,+\infty)\to[0,+\infty)$ be a continuous decreasing function. Then, for every $n>1$, we have 
\begin{equation}\label{eqn:sobolev-1}
\int_0^{+\infty} V(t) \,t^{1/(n-1)} \,dt \leq \frac{n-1}{n}\left( \int_0^{+\infty} V(t)^{(n-1)/n} \,dt \right)^{n/(n-1)}.
\end{equation}
\end{lemma}	
\begin{proof}
Consider the function 
$$v(T):=\int_0^{T} V(t) \,t^{1/(n-1)} \,dt - \frac{n-1}{n}\left( \int_0^{T} V(t)^{(n-1)/n} \,dt \right)^{n/(n-1)}.$$
Taking the derivative in $T$ and using the monotonicity of $V$, we have 	
\begin{align*}
v'(T)&=V(T) T^{1/(n-1)}-V(T)^{(n-1)/n} \left( \int_0^{T} V(t)^{(n-1)/n} \,dt \right)^{1/(n-1)}\\
&\le V(T) T^{1/(n-1)}-V(T)^{(n-1)/n} \left( TV(T)^{(n-1)/n}\right)^{1/(n-1)}=0,
\end{align*}
which concludes the proof since $v(0)=0$.
\end{proof}
	
\begin{proof}[Proof of \cref{thm:sobolev}] We follow \cite[Theorem 3]{BoGi}. We divide the proof in two steps.  For ease of notation write $\Omega := \Omega_u$.

\medskip

\noindent \emph{Step 1.} We first prove \eqref{eqn:sobolev-concl1}, \eqref{eqn:sobolev-concl2} for $f \in C^\infty(\Omega \cap B_1)$.  Let $k$ be so that
\[
\max\left\{ \haus^d(\{ f > k\} \cap B_\gamma), \haus^d( \{ f < k \} \cap B_\gamma) \right\} \leq \frac12\, \haus^d( B_\gamma \cap \Omega),
\]
and let $f_1:= (f-k)^+$ and $f_2:= (k-f)^+$.  Note that by our choice of $k$ we have 
\[
\haus^d( \{f_i > t\} \cap B_\gamma ) \leq \haus^d(\{f_i \leq t \} \cap B_\gamma )\,,\qquad \text{for all $t > 0$, $i = 1, 2$,}
\]
so that, by the Neumann-isoperimetric inequality of \cref{thm:niso} we get
\begin{equation}\label{eqn:neu_iso}
\haus^d( \{ f_i > t \} \cap B_\gamma )^{(d-1)/d}\leq C(d) \,\haus^{d-1}( \del \{ f_i > t \})\,,\qquad \text{for all $t > 0$, $i = 1, 2$.}
\end{equation}
Therefore, by the coarea formula, \eqref{eqn:sobolev-1} and \eqref{eqn:neu_iso}, we have
\begin{align*}
\int_{\Omega \cap B_1} |Df_i|= \int_0^\infty \haus^{d-1}( \del \{ f_i > t \})\,dt
&\geq C(d) \int_0^\infty \haus^d( \{ f_i > t \} \cap B_\gamma )^{\frac{d-1}d} \,dt\\
&\geq C(d) \bigg( \int_0^\infty \haus^d( \{ f_i > t \}\cap B_\gamma)\,t^{\frac1{d-1}} dt \bigg) ^{\frac{d-1}d} \\
&= C(d) \bigg( \int_{\Omega \cap B_\gamma} |f_i|^{\frac{d}{d-1}} \bigg)^{\frac{d-1}d}.
\end{align*}
Since by construction $|f - k| = |f_1| + |f_2|$ and $|Df| = |Df_1| + |Df_2|$, we get \eqref{eqn:sobolev-concl1}.\smallskip

Finally, we notice that \eqref{eqn:sobolev-concl2} follows by the same argument. In fact, if $f$ is supported inside $B_1$, then we can use the isoperimetric inequality of \cref{thm:iso} to do the same computation as above with $f$ in place of $f_i$.

\medskip

\noindent \emph{Step 2.} Conclusion of the proof.  Take $f \in W^{1,1}(\Omega \cap B_1)$.  If $\spt f \subset B_1$, then by \cref{lem:approx} below we can find an approximating sequence $\phi_i \in C^\infty_c(B_1\setminus\sing(u))$ so that $\phi_i \to f$ in $W^{1,1}(\Omega \cap B_1)$ (see \cref{lem:approx} below).  Moreover, since inequality \eqref{eqn:sobolev-concl1} holds for each $\phi_i$, the converegence $\phi_i \to f$ is also strong in $L^{d/(d-1)}(\Omega \cap B_1)$.  We deduce \eqref{eqn:sobolev-concl2}.\smallskip

We prove \eqref{eqn:sobolev-concl1}.  Let $\zeta$ be a smooth cut-off function supported in $B_1$ which is identically equal to $1$ in $B_{\sfrac12}$.  Pick $\phi_i \in C^\infty_c(B_1\setminus\sing(u))$ so that $\phi_i \to f\zeta$ in $W^{1,1}(\Omega \cap B_1)$ (see \cref{lem:approx}).  For each $i$ there is a constant $c_i$ so that
\begin{equation}\label{eqn:sobolev-3}
\bigg( \int_{\Omega \cap B_{\sfrac\gamma2}} |\phi_i - c_i|^{d/(d-1)} \bigg)^{(d-1)/d} \leq C(d) \int_{\Omega \cap B_{\sfrac12} } |D\phi_i|.
\end{equation}
From \eqref{eqn:sobolev-concl2} we also have that $\phi_i \to f\zeta$ in $L^{d/(d-1)}(\Omega \cap B_1)$.

Now if $\haus^d(\Omega \cap B_{\sfrac\gamma2}) = 0$ then \eqref{eqn:sobolev-concl1} trivially holds with $\sfrac\gamma2$ in place of $\gamma$.  Assume therefore that $\haus^d(\Omega \cap B_{\sfrac\gamma2}) = \theta > 0$.  For any $i >> 1$ we use \eqref{eqn:sobolev-concl2}, \eqref{eqn:sobolev-3} to compute
\begin{align*}
\theta^{(d-1)/d} |c_i|\leq ||c_i||_{L^{d/(d-1)}(\Omega\, \cap B_{\sfrac\gamma2})}
&\leq ||\phi_i - c_i||_{L^{d/(d-1)}(\Omega\, \cap B_{\sfrac\gamma2})} + ||\phi_i||_{L^{d/(d-1)}(\Omega\, \cap B_1)} \\
&\leq C(d)||\phi_i||_{W^{1,1}(\Omega\, \cap B_1)}\leq 2C(d) || f\zeta ||_{W^{1,1}(\Omega\, \cap B_1)}.
\end{align*}
Therefore the constants $c_i$ are uniformly bounded independent of $i$ and (after passing to a subsequence) we can assume that $c_i \to c$.  Recalling our definition of $\zeta$, and convergence $\phi_i \to f\zeta$ in $L^{d/(d-1)}(\Omega \cap B_1)$, we get \eqref{eqn:sobolev-concl2} with $\sfrac\gamma2$ in place of $\gamma$.
\end{proof}

In Step 2 of the proof above we used the following approximation theorem.

\begin{lemma}\label{lem:approx}
Let $u \in W^{1,2}(B_1)$ be a minimizer of $J_{B_1}$ and let $f \in W^{1,p}(\Omega_u \cap B_1)$, for some $1 \leq p < 5$.  Then for any $\theta < 1$ we can find a sequence $\phi_i \in C^\infty_c(B_1 \setminus \sing (u))$ so that $\phi_i \to f$ in $W^{1,p}(\Omega_u \cap B_{1-\theta})$.  If $\spt f \subset B_1$, then we can take $\theta = 0$.
\end{lemma}

\begin{proof}
As usual we let $\Omega:= \Omega_u$. For any $k \in \R$, note that $f_k := \min\{ k, \max\{ f, -k \}\} \in W^{1,p}(\Omega \cap B_1)$, and
\[
||f - f_k||_{W^{1,p}(\Omega \cap B_1)}^p \leq \int_{\{|f| > k\}} \left(|f|^p + |Df|^p \right)\to 0
\]
as $k \to \infty$.  Therefore there is no loss in assuming $f \in L^\infty(\Omega \cap B_1)$.

We next claim we can additionally assume that $\spt f \cap B_\eps(\sing(u)) = \emptyset$ for some $\eps > 0$.  Trivially, since $\sing(u) \subset \del\Omega$, we have
\[
\int_{\Omega \cap B_1 \cap B_\eps(\sing(u))} \left(|f|^p + |Df|^p\right) \leq \tau(\eps)^p,
\]
for for some $\tau(\eps) \to 0$ as $\eps \to 0$.

Since $\sing (u)$ has Hausdorff dimension $\leq d-5$ and $p < 5$, for any $\eps > 0$ we can find a finite cover $\{B_{s_i}(y_i)\}_{i=1}^M$ of $\sing(u) \cap \overline{B_1}$ satisfying $\sum_i s_i^{d-p} \leq \eps$ and $y_i \in \sing(u)$.  For each $i$ choose an $\eta_i \in C^\infty$ satisfying
\[
\eta_i \equiv 0 \text{ in } B_{s_i}(y_i), \quad \eta \equiv 1 \text{ outside } B_{2s_i}(y_i), \quad |D\eta_i| \leq 10/s_i
\]
Define $\eta = \inf \eta_i$.  Then $\eta$ is a Lipschitz function satisfying
\[
\spt \eta \cap \sing (u)= \emptyset, \quad \eta \equiv 1 \text{ outside } B_{2\eps}(\sing(u)), \quad |D \eta(x)| \leq \sup_i \frac{10}{s_i} 1_{B_{2s_i}(y_i)}(x).
\]

Now $f \eta \in W^{1,p}(\Omega \cap B_1)$, $\spt (f \eta) \cap \sing(u) = \emptyset$, and
\begin{align*}
||f - f \eta||_{W^{1,p}(\Omega \cap B_1)} 
&\leq \tau(2\eps) + \left(\int_{\Omega \cap B_1} |f D\eta|^p \right)^{1/p} \\
&\leq \tau(2\eps) + c(d) ||f||_{L^\infty(\Omega \cap B_1)} \sum_i s_i^{d-p} \\
&\leq \tau(2\eps) + c(d) ||f||_{L^\infty(\Omega \cap B_1)} \eps \, ,
\end{align*}
which $\to 0$ as $\eps \to 0$.  This proves our claim.

We proceed assuming $\spt f \cap B_\eps(\sing(u)) = \emptyset$, for some $\eps > 0$.  Since $\del\Omega \setminus \sing(u)$ is smooth, after perturbing $B_{1-\theta}$ to a smooth domain $B_{1-\theta/2} \supset U \supset B_{1-2\theta}$, we can assume that $\del(U \cap \Omega)$ is locally-Lipschitz in $B_1 \setminus B_{\eps/4}(\sing(u))$.

Choose a finite cover $\{B_{r_i}(x_i)\}_{i=1}^N$ of $\del (\Omega \cap U) \setminus B_{\eps}(\sing (u))$ such that $x_i \in \del(\Omega \cap U)$, $B_{2r_i}(x_i) \subset B_1 \setminus B_{\eps/2}(\sing (u))$, and each $\del (\Omega \cap U) \cap B_{2r_i}(x_i)$ is a Lipschitz graph.  Pick smooth functions $\zeta_0, \ldots, \zeta_N$ such that
\[
\spt \zeta_i \subset B_{2r_i}(x_i), \quad \spt \zeta_0 \subset \Omega \cap B_1 \setminus B_{\eps/2}(\sing (u)), \quad \sum_{i=0}^N \zeta_i = 1 \text{ on } \Omega \cap U \setminus B_{\eps}(\sing(u)).
\]
By the usual extension/approximation theorems for Sobolev functions applied to each $f \zeta_i$, we can find a sequence of smooth functions $\phi_k \in C^\infty_c(B_1 \setminus \sing(u))$ so that $\phi_k \to f$ in $W^{1,p}(\Omega \cap U)$.  This proves the first assertion of the Lemma, with $2\theta$ in place of $\theta$.  The second assertion follows because $\spt f \subset B_1$ implies $\spt f \subset B_{1-\theta}$ for some $\theta > 0$.
\end{proof}

 \section{De\,Giorgi-Nash-Moser theory}\label{ss:4}

By nowadays standard iteration methods (see e.g. \cite[Theorems 5 and 6]{BoGi}), the inequalities of \cref{ss:4} imply the standard integral/Harnack estimates of De\,Giorgi-Nash-Moser.  For the reader's convenience, in \cref{ss:john} we reproduce a proof (different from \cite{BoGi} and originally due to L. Simon)  of the John-Nirenberg lemma adapted to our setting.

\begin{theorem}[Subsolutions]\label{thm:sub}
Let $u \in W^{1,2}(B_{R_1})$ be a minimizer of $J_{B_{R_1}}$. Suppose $f \in W^{1,2}(\Omega_u \cap B_1)$ is non-negative and satisfies
\begin{equation}\label{eqn:sub-hyp}
\int_{\Omega_u} D f \cdot D\phi \leq 0 
\end{equation}
for all non-negative $\phi \in C^1_c(B_1\setminus \sing(u))$.  Then
\begin{equation}\label{eqn:sub-concl}
\sup_{\Omega_u \cap B_\theta} f \leq c(\theta, p, d) \left( \int_{\Omega_u \cap B_1} f^p \right)^{1/p},
\end{equation}
for all $0 < p < \infty$ and all $\theta < 1$.
\end{theorem}

\begin{proof}
Follows from \eqref{eqn:sobolev-concl2} and \eqref{eqn:sub-hyp} by well-known iteration methods.
\end{proof}

\begin{theorem}[Supersolutions]\label{thm:super}
Let $u \in W^{1,2}(B_{R_1})$ be a minimizer of $J_{B_{R_1}}$. There is a dimensional constant $\gamma>0$ so that if $f \in W^{1,2}(\Omega_u \cap B_1)$ is non-negative and satisfies
\begin{equation}\label{eqn:super-hyp}
\int_{\Omega_u} Df \cdot D\phi \geq 0
\end{equation}
for all non-negative $\phi \in C^1_c(B_1 \setminus \sing(u))$, then
\begin{equation}\label{eqn:super-concl}
\bigg( \int_{\Omega_u \cap B_\gamma} f^p \bigg)^{{1}/p} \leq c(p, d) \inf_{\Omega_u \cap B_\gamma} f\ \qquad\text{for all}\qquad p\in\Big(0,\frac{d}{d-2}\Big).
\end{equation}
%for all $0 < p < d/(d-2)$.
\end{theorem}

\begin{proof}
Follows from \cref{thm:sub} and \cref{lem:john} by a well-known argument.
\end{proof}

 \section{One-sided blow-ups near regular points}\label{ss:5}
 
 In this section we study one-sided blow-ups to $1$-homogeneous minimizers of $J$. 
 
 \begin{assumption}\label{ass:osb} We let $u_0 \in W^{1,2}(B_1)$ be a non-zero minimizer of $J_{B_1}$, and let $u^\mu, v^\mu \in W^{1,2}(B_1)$, $\mu\in \N$, be sequences of functions minimizing $J_{B_1}$, such that 
 \[
 u^\mu \leq  v^\mu\,\text{ in }B_1\,;\quad u^\mu, v^\mu\to u_0\,\text{ in }\,(C^\alpha_{loc} \cap W^{1,2}_{loc})(B_1)\,;\quad u^\mu < v^\mu\,\text{ on } \, \Omega_{u^\mu}\,. %\subset \Omega_{v^\mu}\,.
 \]
\end{assumption}

In this section we will prove the following theorem. The main idea is similar to \cite[Proposition 5.1]{DeJeSh}, however our situation is more general and doesn't follow directly from \cite{DeJeSh}, so we will provide the details of the proof.

\begin{theorem}[One-sided blow-up]\label{thm:osb} Let $u_0, u^\mu, v^\mu$ be as in \cref{ass:osb}.\\ Let the point $p \in \Omega_0$ be fixed, and define %$\lambda_\mu:= v^\mu(p) - u^\mu(p)$, and 
 \[
  \lambda_\mu:= v^\mu(p) - u^\mu(p)\qquad\text{and}\qquad w^\mu := \lambda_\mu^{-1}(v^\mu - u^\mu)\in W^{1,2}_{loc}(B_1) \,.
 \]
Then, there is a function $w \in C^{2,\alpha}(\overline{\Omega}_{u_0} \setminus \sing(u_0)\cap B_1) \cap C^\infty(\Omega_{u_0} \cap B_1)$ so that $w^\mu \to w$ in $C^\infty_{loc}(\Omega_{u_0} \cap B_1)$, and $w$ solves:
 \begin{equation}\label{eqn:lin_1}
 \begin{cases}
 \Delta w = 0 &\text{ in } \Omega_{u_0} \cap B_{1}\\
 D_\nu w + H w = 0 &\text{ on } \reg(u_0)\cap B_{1}\\
 w \geq 0 & \text{ in } \Omega_{u_0} \cap B_{1}\,,
 \end{cases}
 \end{equation}
 where $\nu$ and $H$ denote respectively the outer unit normal and the scalar mean curvature of $\reg(u_0)\subset \del \Omega_{u_0}$. 
\end{theorem}

\begin{remark}
Recall that if $u_0 \in W^{1,2}_{loc}(\R^d)$ is a global-minimizer, then by \cref{lem:H} $H \leq 0$ on $\reg(u_0)$.	
\end{remark}	
Combining \eqref{eqn:lin_1} and the theory developed in the previous sections, we can prove the following key estimate.

\begin{proposition}[Harnack inequality]\label{thm:harnack}
Let $u_0 \in W^{1,2}_{loc}(\R^d)$ be a global minimizer of $J_{\R^d}$, and $w\in C^{2,\alpha}(\overline{\Omega}_{u_0} \setminus \sing(u_0)\cap B_1) \cap C^\infty(\Omega_{u_0} \cap B_1)$ be a solution of \eqref{eqn:lin_1}.  There exist dimensional constants $C,\gamma>0$ such 
 \begin{equation}\label{eqn:lin_2}
 \int_{\Omega_0 \cap B_\gamma} w \leq C \inf_{\Omega_0 \cap B_\gamma} w.
 \end{equation}
\end{proposition}

\noindent The rest of this section is devoted to the proofs of \cref{thm:osb} and \cref{thm:harnack}.

\subsection{Proof of Theorem \ref{thm:osb}} We divide the proof in two steps.

\medskip

\noindent \emph{Step 1.} We start by analyzing the behavior of the blow-up sequence at regular points of the free-boundary. So let $u$ be as in \cref{thm:ac-reg}.  Write $\R^{d+1} = \{ (x', x_d, x_{d+1}) \in \R^{d-1} \times \R \times \R \}$.  Choosing $\eps>0$ sufficiently small in \cref{thm:ac-reg}, by \eqref{eqn:ac3}, we can consider the hodograph transform $x'=y'$ and $y_d=u(x)$, to find a function $$H_u : B_{1-2\eps} \cap \{ y_d \geq 0\} \to \R\,,$$ 
satisfying
\begin{equation}\label{eqn:diff-1}
H_u(x', u(x', x_d)) = x_d, \quad\text{and}\quad u(x', H_u(x', y_d)) = y_{d}\,,
\end{equation}
so that the free boundary of $u$ is given by (the graph of) the trace of $H_u$ over the hyperplane $\{y_d=0\}$. Standard calculations yield 
\begin{equation}\label{eqn:diff-1.1}
\begin{cases}
\sum_{i, j} a_{ij}(D H_u) \, D^2_{ij}H_u = 0 &\text{in } B_{1-2\eps} \cap \{ y_d > 0\} \\
D_d\, H_u= B(D_{1} H_u, \ldots,D_{d-1} H_u)&\text{on }B_{1-2\eps} \cap \{ y_d = 0\}\,,
\end{cases}
\end{equation}
with $a_{ij}, B$ analytic and $a_{ij}(D H_u)$ uniformly elliptic.

\smallskip

Next, suppose that $u_0$ is $\eps/2$-flat in a ball $B_1$ with $\eps$ as in the statement of \cref{thm:ac-reg}. %Then, we can assume ., $u_\mu$ and $v_\mu$ satisfy the hypotheses of \cref{thm:ac-reg} and that $u^\mu \leq v^\mu \in W^{1,2}(B_1)$ minimize $J_{B_1}$, and $u^\mu, v^\mu \to u_0$ in $(C^\alpha_{loc}\cap W^{1,2}_{loc})(B_1)$. 
Then, for $\mu$ sufficiently large, also the functions $u^\mu, v^\mu$ are $\eps$-flat in $B_1$, so we can apply \cref{thm:ac-reg}. Let $H_{u_0}$, $H_u^\mu$ and $H_v^\mu$ be the hodograph transforms of $u_0$, $u_\mu$ and $v_\mu$ on $B_{1-2\eps}^+$; we set for simplicity $H_0:=H_{u_0}$.  Since $u^\mu \leq v^\mu$, we have 
$$H_u^\mu(x',0) \geq H_v^\mu(x',0)\quad\text{for every}\quad (x',0)\in B_{1-2\eps}\cap \{y_d=0\}.$$  From \cref{thm:ac-reg} we can also assume that $H_u^\mu, H_v^\mu\to H_0$ in $C^{3,1}(B_{1-2\eps})$.

Since $a_{ij}, B$ (in \eqref{eqn:diff-1.1}) are analytic functions of $D H_u$, we can use the fundamental theorem of calculus to deduce that the difference 
$$\tilde w^\mu = H_u^\mu - H_v^\mu$$ 
solves a PDE of the form
\begin{equation}\label{eqn:diffeq}
\begin{cases}
\sum_{i, j} \tilde a_{ij} D^2_{ij} \tilde w^\mu = 0 & \text{ in }B_{\sfrac56}^+:=B_{\sfrac56}\cap \{y_d>0\}\,,\\
D_d w^\mu = \sum_i \tilde b_i D_i \tilde w^\mu & \text{ on }B_{\sfrac56}\cap \{y_d=0\}\,,
\end{cases}
\end{equation}
where $\tilde a_{ij}$ is uniformly elliptic and depends analytically on $D H_u^\mu, D^2 H_u^\mu, D H_v^\mu, D^2 H_v^\mu$, and where $\tilde b_i$ are analytic functions of $D H_u^\mu, D H_v^\mu$.  In particular, by \cref{thm:ac-reg}, $\tilde a_{ij}$, $\tilde b_i$ have (respectively) uniform $C^\alpha$ and  $C^{1,\alpha}$ bounds, depending only on the dimension $d$.  Using the Harnack inequality and Schauder theory for strong solutions with oblique boundary conditions (see for instance \cite[Theorem 5.2]{DeJeSh}), we get
\begin{equation}\label{eqn:diff-1.4}
\sup_{B_{\sfrac34}^+} \tilde w^\mu \leq C(d) \inf_{B_{\sfrac34}^+} \tilde w^\mu \qquad \text{and}\qquad ||\tilde w^\mu||_{C^{2,\alpha}(B_{\sfrac12}^+)} \leq C(d) ||\tilde w^\mu||_{L^\infty(B_{\sfrac34}^+)}.\smallskip
\end{equation}
\noindent Next, let $(x', x_d) \in \{ v^\mu > 0 \} \cap B_{\sfrac12}$ and $y_{d} := u^\mu(x', x_d)$.  Then $(x', y_d) \in B_{\sfrac34}^+$ and 
%$(x', y_d) \in B_{\sfrac34}^+$, and (by definition) $v^\mu(x', H_v^\mu(x', y_d)) = y_{d} = u^\mu(x', x_d)$.  We deduce
\begin{align}%\label{eqn:diff-1.2}
v^\mu(x', x_d) - u^\mu(x', x_d)&=v^\mu(x', H_u^\mu(x', y_d)) - y_d\notag\\
&=v^\mu(x', H_u^\mu(x', y_d)) - v^\mu(x', H_v^\mu(x', y_d))\notag\\
&= \int_{H_v^\mu(x', y_d)}^{H_u^\mu(x', y_d)} D_d v^\mu(x', t)\, dt \,. \label{eqn:diff-1.2}
\end{align}
Combined with \eqref{eqn:diff-1.4} and the fact $v^\mu$ satisfies \eqref{eqn:ac3}, the above \eqref{eqn:diff-1.2} implies
\begin{equation}\label{eqn:diff-2}
\sup_{B_{\sfrac12}} (v^\mu - u^\mu) \leq C(d) \big(v^\mu(0, 1/8) - u^\mu(0, 1/8)\big),
\end{equation}
\begin{equation}\label{eqn:diff-3}
||H_u^\mu - H_v^\mu||_{C^{2,\alpha}(B_{\sfrac12}^+)} \leq C(d) \big(v^\mu(0, 1/8) - u^\mu(0, 1/8)\big),
\end{equation}
for a dimensional constant $C(d)>0$.
\smallskip

Let $(x', x_d) \in \Omega_0 \cap B_{\sfrac12}$. Then, for $\mu > 1$ sufficiently large, 
$$(x', x_d) \in \Omega_{u^\mu} \cap B_{\sfrac12} \subset \Omega_{v^\mu}  \cap B_{\sfrac12}.$$ 
As $\mu\to+\infty$, $H_u^\mu(x', y_d) - H_v^\mu(x', y_d) \to 0$ and $v^\mu \to u_0$ smoothly on compact subsets $B_{\sfrac12}\cap\{y_d>0\}$ and $\Omega_0\cap B_{\sfrac12}$, respectively. Thus, for $\mu$ large, we can compute
\begin{align}
&v^\mu(x', x_d) - u^\mu(x', x_d) \nonumber \\
&= \big(H_u^\mu(x', y_d) - H_v^\mu(x', y_d)\big) \,\int_0^1 D_d v^\mu\big(x', H_v^\mu(x', y_d) + s\,(H_u^\mu(x', y_d) - H_v^\mu(x', y_d))\big)\, ds \nonumber \\
&= \big(H_u^\mu(x', y_d) - H_v^\mu(x', y_d)\big) \big(D_d u_0(x', x_d) + \eps_\mu(x', x_d)\big), \label{eqn:diff-4}
\end{align}
for $\eps_\mu(x', x_d) \to 0$. 

\smallskip

By \cref{thm:ac-reg}, we can write
\begin{gather*}
\del \Omega_u^\mu \cap B_{\sfrac34} = \graph_{\del \Omega_0}(\xi^\mu), \quad \del \Omega_v^\mu \cap B_{\sfrac34} = \graph_{\del \Omega_0}(\eta^\mu),
\end{gather*}
where we graph over the normal pointing in the positive $e_d$ direction.  Our convergence $H_u^\mu, H_v^\mu \to H_0$ implies $||\xi^\mu||_{C^{3,1}(B_{\sfrac34})}, ||\eta^\mu||_{C^{3,1}(B_{\sfrac34})} \to 0$ as $\mu \to \infty$.

By elementary geometry, for $x' \in B_{\sfrac12}^{d-1}$ and $y' = x' + H_0(x', 0) \in \del\Omega_0 \cap B_{3/4}$ we can write
\begin{equation}\label{eqn:diff-12}
(\xi^\mu(y') - \eta^\mu(y'))(1+R^\mu(x')) = \frac{H_u^\mu(F^\mu(x')) - H_v^\mu(F^\mu(x'))}{\sqrt{1+|D' H_0(x', 0)|^2}},
\end{equation}
where each $F^\mu : B^{d-1}_{\sfrac12} \to \R^{d-1}$ is a smooth diffeomorphism onto its image, $R^\mu : B_{\sfrac12}^{d-1} \to \R^{d-1}$ is smooth, and
\begin{equation}\label{eqn:diff-13}
||F^\mu - \mathrm{id}||_{C^{2,1}(B_{\sfrac12})} \to 0, \quad ||R^\mu||_{C^{2,1}(B_{\sfrac12})} \to 0,
\end{equation}
and $D' f = \pi_{\R^{d-1}}(D f)$.  Therefore, by \eqref{eqn:diff-2}, \eqref{eqn:diff-3}, \eqref{eqn:diff-12}, \eqref{eqn:diff-13} we have
\begin{equation}\label{eqn:diff-5}
||\xi^\mu - \eta^\mu||_{C^{2,\alpha}(B_{\sfrac12 - \delta_\mu}\cap \del\Omega_0)} \leq C(d) ||H_u^\mu - H_v^\mu||_{C^{2,\alpha}(B_{\sfrac12})}
\end{equation}
where $\delta_\mu \to 0$, and, for any $(x', x_d) \in \del \Omega_0 \cap B_{\sfrac12}$,
\begin{align}\label{eqn:diff-6}
\xi^\mu(x', x_d) - \eta^\mu(x', x_d) 
&= D_d u_0(x', x_d)(H_u^\mu(x', 0) - H_v^\mu(x', 0)) + \eps_\mu'(x', x_d),
\end{align}
where $|v^\mu(0, 1/8) - u^\mu(0, 1/8)|^{-1} \eps_\mu'(x', x_d) \to 0$.

\vspace{3mm}

Assume that $v^\mu(0, 1/8) - u^\mu(0, 1/8) > 0$ for all $\mu$.  Let $\lambda_\mu \in \R$ be any sequence such that
\[
1/\Gamma \leq \lambda_\mu^{-1}(v^\mu(0, 1/8) - u^\mu(0, 1/8)) \leq \Gamma \quad \forall \mu,
\]
for some $\Gamma > 0$.  Define
\[
w^\mu = \lambda_\mu^{-1} (v^\mu - u^\mu), \quad k^\mu = \lambda_\mu^{-1}(H_u^\mu - H_v^\mu), \quad \tau^\mu = \lambda_\mu^{-1}(\xi^\mu - \eta^\mu).
\]

From \eqref{eqn:diff-2}, we have
\begin{equation}\label{eqn:diff-7}
||w^\mu||_{L^\infty(B_{\sfrac12})} \leq c(d, \Gamma).
\end{equation}
By \eqref{eqn:diff-7}, \eqref{eqn:diff-3}, \eqref{eqn:diff-6} after passing to a subsequence we can find $w \in C^\infty(\Omega_0 \cap B_{\sfrac12})$, $k \in C^{2,\alpha}(B_{\sfrac12})$ and $\tau \in C^{2,\alpha}(\del \Omega_0 \cap B_{\sfrac12})$ so that
\begin{gather}
w_\mu \to w \text{ in } C^\infty_{loc}(\Omega_0 \cap B_{\sfrac12}), \quad k_\mu \to k \text{ in } C^{2,\alpha'}(B_{\sfrac12}), \label{eqn:diff-20} \\
\tau^\mu \to \tau \text{ in } C^{2,\alpha'}_{loc}(\del \Omega_0 \cap B_{\sfrac12}) \label{eqn:diff-21}
\end{gather}
for all $\alpha' < \alpha$.  Moreover, from \eqref{eqn:diff-4}, \eqref{eqn:diff-6} we have
\[
w = k D_d u_0 \quad \text{ on } \Omega_0 \cap B_{\sfrac12},
\]
and
\[
\tau = k D_d u_0  \quad \text{ on } \del \Omega_0 \cap B_{\sfrac12}\,.
\]
We deduce that
\begin{equation}\label{eqn:blowup_diff_1}
w \in C^{2,\alpha}(\overline{\Omega_0} \cap B_{\sfrac12}) \quad \text{ and } \quad  w|_{\del \Omega_0 \cap B_{\sfrac12} } = \tau.
\end{equation}

\medskip 
 
 \noindent \emph{Step 2.} Suppose now we have $u_0, u^\mu, v^\mu, \lambda_\mu$, and $w^\mu \in W^{1,2}(B_1)$ as in \cref{ass:osb} and \cref{thm:osb}. Write $\Omega_0 = \Omega_{u_0}$. Fix $U \subset\subset B_1 \setminus \sing (u_0)$.  By \cref{thm:ac-reg}, for $\mu$ sufficiently large we can write  $\del\Omega_{u^\mu} \cap U = \graph_{\del\Omega}(\xi^\mu)$, $\del\Omega_{v^\mu} \cap U = \graph_{\del\Omega}(\eta^\mu)$ with respect to the inner normals.  From \eqref{eqn:diff-7} and the usual Harnack inequality in the interior of $\Omega_0$, we have
 \begin{equation}\label{eqn:diff-22}
 \sup_\mu ||w^\mu||_{L^\infty(U)} < \infty,
 \end{equation}
 and so we can find a non-negative $w \in C^{2,\alpha}(\overline{\Omega_0} \cap U) \cap C^\infty(\Omega_0 \cap U)$ so that $w^\mu \to w$ in $C^\infty_{loc}(\Omega_0 \cap U)$, where we used \eqref{eqn:blowup_diff_1} to obtain the $C^{2,\alpha}$ regularity up to the regular part of the boundary of $\Omega_0$.
 
 Fix $\phi \in C^1_c(U)$.  Since $u^\mu|_{\del\Omega_{u^\mu} \cap U} = 0$ and the outer derivative $D_\nu u^\mu|_{\del\Omega_{u^\mu} \cap U} = -1$ (and the same for $v^\mu$ w.r.t. $\Omega_{v^\mu}$) we compute
 \begin{align*}
 \int \Delta \phi ( v^\mu - u^\mu)
 &= - \int_{\Omega_{v^\mu}} D \phi \cdot D v^\mu + \int_{\Omega_{u^\mu}} D \phi \cdot Du^\mu \\
 &= \int_{\del \Omega_{v^\mu}} \phi - \int_{\del\Omega_{u^\mu}} \phi \\
 &= \int_{\del\Omega_0} \phi(x - \eta^\mu(x) \nu(x)) J \eta^\mu(x) - \phi(x - \xi^\mu(x)\nu(x)) J \xi^\mu(x) 
 \end{align*}
 Here $\nu$ denotes the outer unit normal of $\Omega_0$, and $J\eta^\mu$ is shorthand for the Jacobian of the map $\del\Omega_0 \ni x \mapsto x - \eta^\mu(x) \nu(x)$ (and the same for $\xi^\mu$).
 
 There are functions $\eps_\mu, \eps'_\mu, \eps''_\mu, \eps'''_\mu \to 0$ as $\mu \to \infty$ so that
 \begin{align*}
 &\int \Delta \phi (v^\mu - u^\mu)  \\
 &= \int_{\del\Omega_0} (D_\nu \phi + \eps_\mu)(\xi^\mu - \eta^\mu) (1+\eps'_\mu) + (\phi + \eps''_\mu)(J\eta^\mu - J\xi^\mu) \\
 &= \int_{\del\Omega_0} (D_\nu \phi + \eps_\mu)(\xi^\mu - \eta^\mu) (1 + \eps'_\mu) + (\phi + \eps''_\mu) (H + \eps'''_\mu)(\xi^\mu - \eta^\mu)
 \end{align*}
 where $H = {\rm div}_{\del\Omega_0}(\nu)$ is the mean curvature with respect to the outer normal.
 
 If we divide both sides by $\lambda_\mu$, then by \eqref{eqn:diff-22}, \eqref{eqn:diff-20}, \eqref{eqn:diff-21}, \eqref{eqn:blowup_diff_1} we can take a limit as $\mu \to \infty$ to deduce, using \eqref{eqn:blowup_diff_1}, that
 \[
 \int_{\Omega_0} w  \Delta \phi = \int_{\del\Omega_0}w D_\nu \phi  + H \phi w
 \]

 Since $w$ is $C^2$ up to $\overline{\Omega}_0 \cap U$, and $\Delta w = 0$ in $\Omega_0$, we can integrate by parts to get
 \[
 \int_{\Omega_0} D \phi \cdot D w = - \int_{\del\Omega_0}  H \phi w, \quad \text{ or } \quad \int_{\del\Omega_0} \phi( D_\nu w + H w) = 0.
 \]
 Since $\phi$ is arbitrary we deduce that $w$ satisfies $D_\nu w + H w = 0$ on $\reg(u_0) \cap U$.
 
 \smallskip
 
 Since $U \subset\subset B_1 \setminus \sing (u_0)$ was arbitrary, by a diagonalization argument we deduce that there is a non-negative $w \in C^{2,\alpha}(\overline{\Omega_0} \setminus \sing(u_0)\cap B_1) \cap C^\infty(\Omega_0 \cap B_1)$ solving \eqref{eqn:lin_1} so that $w^\mu \to w$ in $C^\infty_{loc}(\Omega_0 \cap B_1)$. \qed
% \begin{equation}\label{eqn:diff-10}
% \Delta w = 0 \text{ in } \Omega_0 \cap B_{1}, \quad D_\nu w + H w = 0 \text{ on } \reg(u_0)\cap B_{1}, \quad w \geq 0.
% \end{equation}

\subsection{Proof of Proposition \ref{thm:harnack}}  If we let $w_k = \min\{ w, k\}$ for $k \geq 0$, then by \eqref{eqn:lin_1} and \cref{lem:H} we get that
 \[
 \int_{\Omega_0} D \phi \cdot D w_k \geq 0 \quad \forall \phi \in C^1_c(B_1 \setminus \sing(u_0)) \text{ non-negative}.
 \]
 By \cref{lem:approx} we can replace $\phi$ with $(w_k+1)^{-1} \zeta^2$ for any fixed $\zeta \in C^1_c(B_1 \setminus \sing(u_0))$ non-negative, to get
 \[
 \int_{\Omega_0} (w_k+1)^{-2} |Dw_k|^2 \zeta^2 \leq 4 \int_{\Omega_0} |D\zeta|^2 .
 \]
 Arguing as in the proof of \cref{lem:approx}, we can find a sequence $\zeta_i \in C^1_c(B_1\setminus \sing(u_0))$ so that $\displaystyle\int_{\Omega_0} |D\zeta_i|^2 \to 1$ and $\zeta_i \to 1$ a.e. on $B_{\sfrac12}$.  Therefore we get the bound
 \[
 \int_{\Omega_0 \cap B_{\sfrac12}} |Dw_k|^2 \leq 4 (k+1)^2,
 \]
 and hence $w_k \in W^{1,2}(\Omega_0 \cap B_{\sfrac12})$.
 
 By \cref{thm:super} we deduce there are dimensional constants $C,\gamma>0$ so that
 \[
 \int_{\Omega_0 \cap B_\gamma} w_k \leq C \inf_{\Omega_0 \cap B_\gamma} w_k,
 \]
 and hence, taking $k \to \infty$, we get \eqref{eqn:lin_2}. \qed

 \section{Proof of Theorem \ref{thm:main}}\label{ss:6}

Here we put together the various ingredients of the previous sections to prove \cref{thm:main}.  The argument follows \cite{Si_max}.  As outlined in the introduction, we first show that violating the strict maximum principle of \cref{thm:main} effectively implies there is a point where the blow-ups of $u, v$ agree.  By a suitable blow-up argument, we can obtain a positive Jacobi field $w$ that decays like $O(r)$, which will contradict the Harnack theory which says $w$ must be uniformly bounded below.

\begin{lemma}\label{lem:theta}
	There is a positive dimensional constant $\theta_0$ such that if $u \in W^{1,2}_{loc}(\R^d)$ is a non-zero $1$-homogenous global minimizer of $J_{\R^d}$, then
	\[
	\{ x \in \Omega_u : d(x, \del \Omega_u) > \theta_0 |x| \} \neq \emptyset.
	\]
\end{lemma}

\begin{proof}
	If the Lemma failed, we could find a sequence $u_i$ of $1$-homogenous minimizers such that
	\begin{equation}\label{eqn:theta-1}
	d(x, \del\Omega_{u_i} ) \leq (1/i) |x| \quad \forall x \in \Omega_{u_i}.
	\end{equation}
	Passing to a subsequence, we can assume there is a $1$-homogenous minimizer $u_0$ so that $u_i \to u_0$ in $C^\alpha_{loc}$, by \cref{lem:comp}.  Since $0 \in \del \Omega_{u_0}$, $\Omega_{u_0}$ is a non-empty open set containing some ball $B_\eps(p)$ with $|p| = 1$.  But then we must have $\Omega_{u_i} \supset B_{\eps/2}(p)$ for $i >> 1$, contradicting \eqref{eqn:theta-1} for $i > 2/\eps$.
\end{proof}

\begin{proof}[Proof of \cref{thm:main}]
	Assume that $0 \neq u \leq v$ and $\reg(u) \cap \reg(v) = \emptyset$, but $\del \Omega_u\cap \del\Omega_v \cap U \neq \emptyset$.  We aim to obtain a contradiction.  Note that, since $\reg(u)$ is dense in $\del\Omega_u \cap U$, the interior maximum principle implies $u < v$ on $\Omega_u \subset \Omega_v$.
	
	\medskip
	
	\noindent\emph{Step 1. Dimension reduction.} We claim that we can assume that $U = B_1$ and $0 \in \del \Omega_u \cap \del \Omega_v$ and $u, v$ have the same tangent cone at $0$, in the sense that for any $r_i \to 0$, there is a subsequence $r_i'$ and a $1$-homogenous minimizer $u_0$ so that $u_{0, r_i'} \to u_0$ and $v_{0, r_i'} \to u_0$.
	
	\smallskip
	
	Pick $p \in \del \Omega_u \cap \del \Omega_v \cap U$.  If $u, v$ have the same tangent cone at $x$, in the sense described above, then we can replace $u, v$ with $u_{p, 1-|p|}$, $v_{p, 1-|p|}$ and the claim is proved.  Otherwise, there are a sequence $r_i \to 0$, and $1$-homogenous minimizers $u_0 \leq v_0$, $u_0 \neq v_0$ so that $u_{p, r_i} \to u_0$ and $v_{p, r_i} \to v_0$.
	
	If (after a rotation) $u_0 = x_d^+$ or $v_0 = x_d^+$, then necessarily by domain monotonicity of eigenvalues in the sphere $u_0 = v_0 = x_d^+$.  In this case $p \in \reg(u)\cap \reg (v)$, which contradicts our hypothesis that $\reg(u) \cap \reg(v) = \emptyset$.
	
	So neither $u_0, v_0$ is linear.  Since $u_0 \leq v_0$ are $1$-homogenous, after replacing $u_0$ with $u_0 \circ Q$ for some rotation $Q \in \mathrm{SO}(d)$, we can assume there is an $p_0 \in \del \Omega_{ u_0} \cap \del\Omega_{ v_0} \cap \del B_1$.  If $p_0 \in \reg(u_0) \cap \reg(v_0)$, then since $\Omega_{u_0}$, $\Omega_{v_0}$ are connected (\cref{lem:conn}) the Hopf maximum principle implies $u_0 = v_0$, which is a contradiction.  Therefore, by the same argument as in the previous paragraph, we must have $p_0 \in \sing(u_0) \cap \sing(v_0)$.
	
	If $u_0, v_0$ have the same tangent cone at $p_0$, then as before replace $u, v$ with $(u_0)_{p_0, 1/2}$, $(v_0)_{p_0, 1/2}$ to establish our claim.  Otherwise, we can repeat the previous argument, blowing up $u_0, v_0$ at $p_0$, to obtain new $1$-homogenous minimizers $u_1, v_1$, and a $p_1 \in \del \Omega_{u_1} \cap \del \Omega_{ v_1} \cap\del B_1$ satisfying the same hypotheses as $u_0, v_0, p_0$, but with a $1$-dimensional line of translational symmetry.
	
	By a standard dimension reduction argument, as long as the current $1$-homogenous minimizers $u_k \leq v_k$, $u_k \neq v_k$ do not have the same tangent cone at $p_k$, we can blow-up again to obtain new $1$-homogenous minimizers with an extra dimension of translational symmetry.  Since every $u_k, v_k$ must be non-linear, this process must stop before $u_k, v_k$ have $(d-1)$-dimensions of translational symmetry.  This proves our claim.\medskip
	
	\noindent\emph{Step 2. Construction and decay of the linearized solution $w$.}  Fix $\theta = \theta_0/2$, for $\theta_0$ as in \cref{lem:theta} and set
	\[
	\Omega_\theta := \big\{ x \in \Omega_u : d(x, \del \Omega_u ) > \theta |x| \big\},
	\]
	so by construction $\Omega_\theta \subset \Omega_u \subset \Omega_v$.  
	
Since $u$ and $v$ have the same tangent cone at $0$, we get that 
	\begin{equation}\label{eqn:m-0.5}
	\sup_{\Omega_\theta \cap \del B_r} r^{-1} (v - u) \equiv \sup_{(r^{-1} \Omega_\theta) \cap \del B_1} (v_{0, r} - u_{0, r}) \to 0 \quad \text{ as } r \to 0.
	\end{equation}
As a consequence, for every $R>0$, the supremum  
$$\sup_{r\in(0,R]}\left( \sup_{\Omega_\theta \cap \del B_r} r^{-1}(v - u) \right),$$
is a maximum achieved at some radius $r\in(0,R]$. This implies that we can select a sequence $r_i \to 0$ so that
	\begin{equation}\label{eqn:m-5}
	\sup_{r \leq r_i} \left( \sup_{\Omega_\theta \cap \del B_r} r^{-1}(v - u) \right) \leq 2 \sup_{\Omega_\theta \cap \del B_{r_i}} r_i^{-1}(v - u).
	\end{equation}
	Passing to a subsequence, by \cref{lem:comp} we can assume there is a $1$-homogenous minimizer $u_0$ so that $u_{0, r_i} \to u_0$ and $v_{0, r_i} \to u_0$ in $C^\alpha_{loc}$, and the free-boundaries converge in the local Hausdorff distance.  Fix a point $p \in \Omega_{u_0} \cap \del B_1$, and define $$\lambda_i := v_{0, r_i}(p) - u_{0, r_i}(p) > 0.$$
	Write $\Omega_0:= \Omega_{u_0}$.  By \cref{thm:osb}, applied in $B_2$ rather than $B_1$, we can find a non-negative function $w \in C^2(\overline{\Omega_0} \setminus \sing(u_0) \cap B_2) \cap C^\infty(\Omega_0 \cap B_2)$ satisfying \eqref{eqn:lin_1} and so that the rescaled functions
	\begin{equation}\label{eqn:m-2}
	\lambda_i^{-1}(v_{0, r_i} - u_{0, r_i}) \to w\quad\text{in}\quad C^\infty_{loc}(\Omega_0 \cap B_2).
	\end{equation}
By our normalization, $w(p) = 1$, and so, since $\Omega_0$ is connected (by \cref{lem:conn}), 
$$w > 0\quad\text{on}\quad\Omega_0 \cap B_2.$$ 
	For a number $\theta'>0$, we will use the notation 
	$$\Omega_{0, \theta'} := \big\{ x \in \Omega_0 : d(x, \del\Omega_0) > \theta' |x| \big\}.$$ 
By the convergence of the blow-up sequence $u_{0, r_i}$ to $u_0$, we have that
\begin{equation}\label{eqn:m-3}
(r_i^{-1} \Omega_\theta) \cap \del B_1 \subset \Omega_{0, \theta/2} \cap \del B_1\,,
\end{equation}
 for $i$ large enough. Analogously, for any $\rho > 0$ and $i$ sufficiently large
\begin{equation}\label{eqn:m-4}
(r_i^{-1} \Omega_\theta) \cap B_{2-\rho} \setminus B_\rho \supset \Omega_{0, 2\theta} \cap B_{2-\rho} \setminus B_\rho.
\end{equation}
Now, our choice of $r_i$ in \eqref{eqn:m-5}, combined with \eqref{eqn:m-2}, \eqref{eqn:m-3}, \eqref{eqn:m-4}, implies that
\begin{equation}\label{eqn:m-6}
\sup_{\Omega_{0, 2\theta} \cap \del B_r} r^{-1} w \leq 4 \sup_{\Omega_{0, \theta/2} \cap \del B_1} w \qquad\text{for all}\qquad r \leq 1.	
\end{equation}
Since $\Omega_{0, 2\theta} \neq \emptyset$ (and is obviously dilation-invariant), we get that 
\begin{equation}\label{eqn:m-6-1}
\inf_{\Omega_0\cap\del B_r}w \leq Cr\qquad\text{for all}\qquad r \leq 1,	
\end{equation}
for some constant $C>0$.\medskip

\noindent\emph{Step 3. Harnack inequality and conclusion of the proof.} By \cref{thm:harnack} we have 
	\begin{equation}\label{eqn:m-7}
\inf_{B_\gamma \cap \Omega_0} w \geq \frac{1}{C} \int_{B_\gamma \cap \Omega_0} w > 0
\end{equation}
	with $C, \gamma$ positive dimensional constants, which clearly contradicts \eqref{eqn:m-6-1}.
\end{proof}

\begin{proof}[Proof of \cref{cor:main}]
A direct consequence of \cref{thm:main}, the Hopf maximum principle, and the connectivity of $\Omega_v$.
\end{proof}

\begin{proof}[Proof of \cref{cor:main2}]
First observe that if $\Omega'$ is any connected component of $\Omega_v$, then $v|_{\del U \cap \overline{\Omega'}}$ cannot be identically zero.  For otherwise, we would have $v|_{\Omega'} \in W^{1,2}_0(U)$, and hence by replacing $v$ with $v ' = v \cdot 1_{U \setminus \Omega'}$ we would have $v' - v \in W^{1,2}_0(U)$ and $J_U(v') < J_U(v)$, contradicting minimality of $v$.

Now by \cref{thm:main}, if the conclusion of \cref{cor:main2} failed we would necessarily have $u = v$ on some connected component $\Omega'$ of $\Omega_v$.  But then on some subset $\Gamma \subset \del U \cap \overline{\Omega'}$ of positive $\haus^{d-1}$-measure we would have $0 < u = v$, contradicting our hypothesis.
\end{proof}

\section{Proof of Theorem \ref{thm:fol}} \label{ss:fol}

Our proof follows the same blow-up principle as \cite{HaSi, Wa, DeJeSh}, which is to find a sequence of minimizers $v^\mu$ of $J_{B_1}$ lying to one side of $u_0$, argue that $v^\mu \to u_0$ but $d(0, \Omega_{v^\mu}) > 0$, and then take a limit of a suitable sequence of dilates $v^\mu_{0, r_\mu} \to \uu$.  The key simplification observed by \cite{Wa} is to prove the ``radial graph'' property before blowing-up rather than after, and thereby avoid having to understand the precise asymptotics of the limit $\uu$ (at the ``expense'' of having to know $C^0$ regularity of $v^\mu$ up to $\del B_1$).

\begin{proof}[Proof of \cref{thm:fol}]
Fix $\gamma < 1$, and let $v^\gamma$ minimize $J_{B_1}$ subject to $v^\gamma|_{\del B_1} = \gamma u_0|_{\del B_1}$ (of course $J_{B_1}(\gamma u_0) < \infty$ since $u_0 \in W^{1,2}_{loc}$).  Since $v^\gamma|_{\del B_1} \leq u_0|_{\del B_1}$ and $u_0$ is minimizing, after replacing $v^\gamma$ with $\min\{ v^\gamma, u_0\}$ there is no loss in assuming $v^\gamma \leq u_0$.  By \cref{l:continuity}, $v^\gamma \in C^0(\overline{B_1})$.

We firstly claim that $v^\gamma \leq \gamma u_0$ also.  To see this, observe that $\Omega_{\gamma u_0} = \Omega_{u_0}$ and $v^\gamma \leq u_0$, and hence if $U' = \{ v^\gamma > \gamma u_0\}$ then $U' \subset \Omega_{\gamma u_0}$ and $(v^\gamma - \gamma u_0)^+ \in W^{1,2}_0(U')$ and $\Delta (v^\gamma - \gamma u_0) = 0$ in $U'$.  Therefore the weak maximum principle for harmonic functions implies $(v^\gamma - \gamma u_0)^+ = 0$, proving our claim.

Now $D_\nu(\gamma u_0) = -\gamma \neq -1$ on $\reg(\gamma u_0) \equiv \reg(u_0)$, and so $\reg(v^\gamma) \cap \reg(u_0) = \emptyset$.  By \cref{thm:main} (applied to $v^\gamma$ and $u_0$) we must have $\del\Omega_{v^\gamma} \cap \del \Omega_{u_0} \cap B_1 = \emptyset$.  Together with the interior maximum principle we deduce that $v^\gamma < \gamma u_0$ on $\overline{\Omega_{v^\gamma}} \cap B_1$.  In particular, since $\gamma u_0$ is $1$-homogenous, we have
\begin{equation}\label{eqn:fol-1}
v^\gamma_{0, r} < \gamma u_0 \equiv v^\gamma \quad \text{ on } r^{-1} \overline{\Omega_{v^\gamma}} \cap \del B_1, \quad \forall r < 1.
\end{equation}

We secondly claim that $v^\gamma_{0,r} \leq v^\gamma$ in $B_1$ for all $r < 1$.  Since $d(0, \Omega_{v^\gamma}) > 0$, this is trivially true for all $r$ sufficiently small.  If $r_*$ is the largest radius so that $v^\gamma_{0,r} \leq v^\gamma$ on $B_1$ for all $r < r_*$, then necessarily since $v^\gamma \in C^0(\overline{B_1})$ we must have $v^\gamma_{0, r_*} \leq v^\gamma$ on $B_1$, and there must be an $x \in \overline{B_1} \cap r_*^{-1} \overline{\Omega_{v^\gamma}}$ for which $v^\gamma_{0, r_*}(x) = v^\gamma(x)$.  By \eqref{eqn:fol-1} and  \cref{cor:main2}, this is a contradiction unless $r_* = 1$.

For a fixed $x \in B_1$, our second claim implies $r^{-1} v^\gamma(rx) \leq v^\gamma(x)$ for all $r \leq 1$.  Therefore at any point $x$ where $Dv^\gamma$ exists we must have
\[
0 \leq \frac{d}{dr}\Big|_{r = 1} r^{-1} v^\gamma(rx) = -v^\gamma(x) + x \cdot Dv^\gamma(x).
\]

We thirdly claim that $v^\gamma \to u_0$ in $W^{1,2}(B_1)$ as $\gamma \to 1$.  Otherwise, by standard compactness there would be $\gamma_i \to 1$ so that $v^{\gamma_i} \to v$ for some minimizer $v \in W^{1,2}(B_1)$ satisfying $v|_{\del B_1} = u_0|_{\del B_1}$ and $v \leq u_0$ but $v \neq u_0$.  But since $u_0$ is the unique minimizer of $J_{B_1}$ for its boundary data (see e.g. \cite[Lemma 2.5]{DeJeSh}), this is a contradiction, and proves our third claim.

For each $\gamma < 1$ we have $r_\gamma := d(0, \Omega_{v^\gamma}) > 0$, and from our third claim we have $r_\gamma \to 0$.  We can therefore find a sequence $\gamma_i \to 1$ so that the functions $v^{\gamma_i}_{0, r_{\gamma_i}}$ converge in $(W^{1,2}_{loc} \cap C^\alpha_{loc})(\R^d)$ to some global minimizer $\underline{u}$ satisfying $\uu \leq u$, $d(0, \Omega_{\uu}) = 1$, and
\begin{equation}\label{eqn:fol-2}
-\uu(x) + x \cdot D\uu(x) \geq 0 \quad \text{$\leb^d$-a.e.} x \in \R^d.
\end{equation}
This $\uu$ is our required solution, satisfying \cref{thm:fol}:\ref{item:fol-1},\ref{item:fol-2}.  We now show $\uu$ satisfies the other asserted properties.

We prove $\sing(\uu) = \emptyset$ (i.e. \cref{thm:fol}:\ref{item:fol-4}).  To see this, observe that if $x \in \del \Omega_{\uu}$, then for $r$ sufficiently small \eqref{eqn:fol-2} implies
\[
-r \uu_{x, r}(y) + (x + ry) \cdot D \uu_{x, r}(y) \geq 0 \quad \text{$\leb^d$-a.e. } y \in B_1.
\]
Now if $w$ is any tangent solution to $\uu$ at $x$, then $w$ is a $1$-homogeneous global minimizer of $J_{\R^d}$ satisfying
\begin{equation}\label{eqn:fol-3}
x \cdot D w(y) \geq 0 \quad \text{$\leb^d$-a.e. } y \in \R^d.
\end{equation}
\eqref{eqn:fol-3} implies that $\Omega_{w} \subset \{ y : y \cdot x \geq 0 \}$, and hence we must have $\Omega_w = \{ y : y \cdot x \geq 0 \}$ and $w(y) = (y \cdot x)^+$.  This proves $x \in \reg(\uu)$.

We prove \eqref{eqn:fol-2} holds with $> 0$ in place of $\geq 0$ (i.e. \cref{thm:fol}:\ref{item:fol-3}).  This follows because $w(x) := -\uu(x) + x \cdot D\uu(x)$ is a non-negative Jacobi field on $\Omega_u$, i.e. $w$ satisfies
\[
\Delta w = 0 \text{ in } \Omega_{\uu}, \quad D_\nu w + H w = 0 \text{ on } \del \Omega_{\uu}, \quad w \geq 0,
\]
where $H$ is the mean curvature scalar of $\del\Omega_{\uu}$ w.r.t. the outer unit normal $\nu$.  Non-negativity is obvious, and harmonicity is an easy computation.  The boundary condition follows because along $\del \Omega_{\uu}$ we have
\[
D\uu = -\nu, \quad D^2_{\nu,\nu} \uu = H, \quad D^2_{\nu, e} \uu = 0 \text{ if $e \perp \nu$}.
\]
Now the Harnack inequality of \cref{thm:harnack} implies that either $w \equiv 0$, or $w > 0$ on $\overline{\Omega_{\uu}}$.  But $w$ cannot be identically zero as this would contradict (e.g.) the fact that $d(0, \Omega_{\uu}) = 1$.

We next prove that $\uu_{0, r} \to u_0$ as $r \to \infty$ (i.e. \cref{thm:fol}:\ref{item:fol-5}).  Take any sequence $r_i \to \infty$. Passing to a subsequence we can assume $\uu_{0, r_i} \to u_0'$ for some $1$-homogenous minimizer $u_0' \leq u_0$.  But now by eigenvalue monotonicity for domains in the sphere, we must have $u_0 = u_0'$.  Since the sequence $r_i$ is arbitrary, this proves our assertion.

Lastly, the fact that the dilations of $\del\Omega_{\uu}$ foliate $\Omega_{u_0}$ by smooth, analytic hypersurfaces, which are radial graphs, follows directly from the properties  \ref{item:fol-1}--\ref{item:fol-5}.

The construction of $\ou$ is essentially the same.  Here we take $\gamma > 1$, and define $v^\gamma$ as before.  The same arguments imply that $v^\gamma \geq \gamma u_0$ on $B_1$, and $v^\gamma_{0, r} \geq v^\gamma$ on $B_1$ for every $r < 1$, and hence
\[
-v^\gamma(x) + x \cdot Dv^\gamma(x) \leq 0 \quad \text{$\leb^d$-a.e. } x \in B_1.
\]
Taking an appropriate sequence $\gamma_i \to 1$ and $r_{\gamma_i} = d(0, \Omega_{v^{\gamma_i}}) \to 0$, we can take a limit of $v^{\gamma_i}_{0, r_{\gamma_i}}$ to obtain a global minimizer $\ou \geq u$.  The rest of the argument proceeds as in the case of $\uu$, except using the Jacobi field $-w$ in place of $w$.
\end{proof}

\appendix

\section{John-Nirenberg lemma}\label{ss:john}

We provide here a self-contained proof in our setting of the John-Nirenberg-type lemma used in proving \cref{thm:super}.  The proof is a very (very) minor modification of a proof due to L. Simon.  We reproduce it here for the convenience of the reader.

\begin{lemma}\label{lem:john}
Under the same hypotheses as in \cref{thm:super}, there is a dimensional constant $\gamma(d) > 0$ so that
\begin{equation}\label{eqn:john-concl}
\bigg( \int_{\Omega \cap B_\gamma} f^p \bigg) \bigg( \int_{\Omega \cap B_\gamma} f^{-p} \bigg) \leq c(d, p) \qquad\text{for all}\qquad 0<p<\frac{d}{d-2}.
\end{equation}
\end{lemma}

\begin{proof}
Let $\Omega:=\Omega_u$ and let $\eps > 0$ be fixed.  For $\zeta \in C^1_c(B_1 \setminus \sing(u))$ non-negative, note that $\phi = (f+\eps)^{-1} \zeta^2 \in W^{1,2}(\Omega \cap B_1)$, and is supported in $B_1$.  Therefore we can approximate $\phi$ in $W^{1,2}(B_1)$ by admissible test functions, and from \eqref{eqn:super-hyp} get
\[
\int_{\Omega} -(f+\eps)^{-2} |Df|^2 \zeta^2 + (f+\eps)^{-1} 2 \zeta Df \cdot \zeta \geq 0\,,
\]
and hence
\begin{equation}\label{eqn:john-1}
\int_{\Omega} (f+\eps)^{-2} |Df|^2 \zeta^2 \leq 4 \int_\Omega |D\zeta|^2.
\end{equation}
For $\lambda \in \R$ to be determined later, define $w := \log(f+\eps) - \lambda$.  Then, \eqref{eqn:john-1} gives
\begin{equation}\label{eqn:john-2}
\int_\Omega |Dw|^2 \zeta^2 \leq 4 \int_\Omega |D\zeta|^2 \quad \forall \zeta \in C^1_c(B_1 \setminus \sing(u)) \text{ non-negative}.
\end{equation}
By the same approximation argument as in \cref{lem:approx}, we deduce that $w \in W^{1,2}(\Omega \cap B_r)$ for all $r < 1$.  In particular, if $w_k = \min\{ k, \max \{ -k, w\}\}$, then $|w_k|^p \in W^{1,2}(\Omega \cap B_r)$ for any $p \geq 0$, $r < 1$.  Using \eqref{eqn:sobolev-concl1}, Holder's inequality, and \eqref{eqn:john-2}, we can choose (and fix) a $\lambda$ so that
\begin{align}
\int_{\Omega \cap B_{\gamma/2}} |w|^{n/(n-1)} \leq c(d) \int_{\Omega \cap B_{1/2}} |Dw| \leq c(d). \label{eqn:john-2.5}
\end{align}

\noindent Take $\phi \in C^1_c(B_{\gamma/2}, [0, 1])$, $p \geq 2$, $\displaystyle\beta = \frac1{\chi - 1}$ and $\alpha = 2\beta + 2$.  From \eqref{eqn:sobolev-concl2} we have
\begin{align}
\left( \int_\Omega |w_k|^{2p\chi} \phi^{2\alpha p \chi - \beta \chi} \right)^{1/\chi} 
%&\leq c(n) \int_\Omega |D( |w_k|^p \phi^{\alpha p -\beta})|^2 \\
&\leq c(d) p^2 \int_\Omega |w_k|^{2p-2} |Dw|^2 \phi^{2\alpha p - 2\beta}\notag\\
&\qquad  + c(d, \chi) p^2 \int_\Omega |w_k|^{2p} \phi^{2\alpha p - 2\beta - 2}. \label{eqn:john-3}
\end{align}
On the other hand, replace $\zeta$ with $|w_k|^{p-1} \phi^{\alpha p - \beta}$ in \eqref{eqn:john-2}, and obtain
\begin{align}
\int_\Omega |w_k|^{2p-2} |Dw_k|^2 \phi^{2\alpha - 2\beta} 
&\leq 8 p^2 \int_\Omega |w_k|^{2p-4} |Dw_k|^2 \phi^{2\alpha p - 2 \beta}\notag\\
&\qquad + c(\chi) p^2 \int_\Omega |w_k|^{2p-2} \phi^{2\alpha p - 2\beta - 2}. \label{eqn:john-4}
\end{align}
Using the interpolation $a^\mu b^{1-\mu} \leq \mu a + (1-\mu) b$ for $a, b \geq 0$, $\mu \in (0, 1)$, we have
\begin{equation}
p^2 |w_k|^{2p-4} \leq (1/16) |w_k|^{2p-2} + 16^p p^{2p}. \label{eqn:john-4.2}
\end{equation}
Therefore, combining \eqref{eqn:john-3}, \eqref{eqn:john-4}, \eqref{eqn:john-4.2}, \eqref{eqn:john-2.5}, we get
\begin{align}
\left( \int_\Omega |w_k|^{2p\chi} \phi^{2\alpha \chi - 2\beta\chi} \right)^{1/\chi} 
&\leq c(d)^p p^{2p} \int_{\Omega \cap B_{\gamma/2}} |Dw|^2 + c(d, \chi) p^2 \int_\Omega |w_k|^{2p-2}\phi^{2\alpha \chi - 2\beta - 2} \nonumber \\
&\leq c(d)^p p^{2p} + c(d, \chi) p^2 \int_\Omega |w_k|^{2p-2} \phi^{2\alpha p - 2\beta - 2}.
\end{align}
Recall that by our choice of $\beta$ we have $\beta \chi = \beta + 1$.  Recall also that $(a + b)^\mu \leq a^\mu + b^\mu$ for $a, b \geq 0$ and $\mu \in [0, 1]$.  Defining the measure $d\eta = \phi^{-2\beta \chi} dx \equiv \phi^{-2\beta - 2} dx$, we deduce
\begin{equation}\label{eqn:john-5}
\left( \int_\Omega |w_k|^{2p\chi} \phi^{2\alpha p \chi} d\eta \right)^{1/2p\chi} \leq c(d) p + c(d, \chi)^{1/p} p^{1/p} \left( \int_\Omega |w_k|^{2p} \phi^{2\alpha p} d\eta \right)^{1/2p}.
\end{equation}
For any $\delta \in (0, 1)$ and non-negative measurable $F$, we have by Holder's inequality
\begin{align}
\left( \int_\Omega F^{2p} d\eta \right)^{1/2p} 
%&\leq \left( \int_\Omega f^{2p\delta q} d\eta \right)^{1/2qp} \left( \int_\Omega f^{2p(1-\delta) q/(q-1)} d\eta \right)^{(q-1)/2qp} \\
&\leq \left( \int_\Omega F^{2p\chi} d\eta \right)^{\delta/2p \chi} \left( \int_\Omega F^{2p(1-\delta) \chi/(\chi-\delta)} d\eta \right)^{(\chi-\delta)/2\chi p} \label{eqn:john-6}
\end{align}
Since the map $\delta \mapsto 2p(1-\delta)\chi/(\chi-\delta)$ takes the value $2p \geq n/(n-1)$ when $\delta = 0$ and $0$ when $\delta = 1$, we can choose a $\delta=\delta(p, \chi)$ so that $$2p(1-\delta)\chi/(\chi-\delta) = n/(n-1).$$
Now combine \eqref{eqn:john-5}, \eqref{eqn:john-6}, \eqref{eqn:john-2.5} with $p = 2$, $F = |w_k| \phi^{\alpha}$, and $\delta(p, \chi)$ as in the previous paragraph to get
\begin{align}
\bigg( \int_\Omega |w_k|^{4\chi} \phi^{4\alpha \chi} d\eta \bigg)^{(1-\delta)/4\chi} 
&\leq c(d) + c(d, \chi) \bigg( \int_\Omega |w_k|^{n/(n-1)} \phi^{\alpha n/(n-1) - 2\beta - 2} dx  \bigg)^{(\chi-\delta)/4\chi} \nonumber \\
&\leq c(d) + c(d, \chi) \bigg( \int_{\Omega \cap B_{\gamma/2}} |w|^{n/(n-1)} dx \bigg)^{(\chi-\delta)/4\chi} \nonumber \\
&\leq c(d, \chi). \label{eqn:john-7}
\end{align}
(Break into two cases: either $\int_\Omega F^{2p\chi} d\eta \geq 1$ or $\leq 1$.)

For $\nu = 1, 2, \ldots$, define
\[
\Psi(\nu) = \left( \int_\Omega |w|^{4\chi^\nu} \phi^{4\alpha \chi^\nu} d\eta \right)^{1/4\chi^\nu}
\]
From \eqref{eqn:john-7}, taking $k \to \infty$, we have $\Psi(1) \leq c(d, \chi)$.  From \eqref{eqn:john-5} we have
\[
\Psi(\nu + 1) \leq c \chi^{\nu} + c^{\chi^{-\nu}} \chi^{\nu \chi^{-\nu}} \Psi(\nu)
\]
for $c = c(d, \chi)$.  Now
\[
\prod_{\mu=0}^\infty c^{\chi^{-\mu}} \chi^{\mu \chi^{-\mu}} \leq c(d, \chi),
\]
and so we have
\[
\Psi(\nu) \leq \sum_{\mu=1}^\nu c \chi^\mu \leq c(d, \chi) \chi^\nu.
\]
Recalling that $4\alpha \chi^\nu - 2\beta - 2 > 0$ for all $\nu$, we get by Holder's inequality
%\[
%\left( \int_{\Omega \cap B_{\gamma/2}} |w|^{4\chi^\nu} dx \right)^{1/4\chi^\nu} \leq c(n, \chi) \chi^\nu \quad \forall \nu = 1, 2, \ldots 
%\]
%By Holder's inequality we deduce that
\[
\bigg( \int_{\Omega \cap B_{\gamma/2}} |w|^j dx \bigg)^{1/j} \leq c(d, \chi) j \quad \forall j = 1, 2, \ldots
\]
And hence, using Stirling's approximation and ensuring $\delta \leq 1/2e$ we have
\begin{align*}
\int_{\Omega \cap B_{\gamma/2}} e^{\delta |w|} dx \leq \sum_{j=0}^\infty \int_{\Omega \cap B_{\gamma/2}} \delta^j |w|^j/j! \leq c \sum_{j=0}^\infty (\delta j)^j / j!  \leq c(d, \chi).
\end{align*}
Therefore
\begin{equation}\label{eqn:john-30}
\bigg( \int_{\Omega \cap B_{\gamma/2}} (f+\eps)^\delta dx \bigg) \bigg(  \int_{\Omega \cap B_{\gamma/2}} (f+\eps)^{-\delta} dx \bigg) \leq c(d, \chi)^2
\end{equation}
and, taking $\eps \to 0$, by the montone convergence theorem we get \eqref{eqn:john-concl} for $p \leq 1/2e$ and $\gamma/2$ in place of $\gamma$.

To prove \eqref{eqn:john-concl} for all $0 < p < \chi$ Simon argues as follows.  For $\theta < 1$, $\theta \neq 0$, and $\zeta \in C^1_c(B_1)$, we can plug in $(f+\eps)^{\theta-1} \zeta^2$ into \eqref{eqn:super-hyp} to obtain
\[
(1-\theta) \int_\Omega (f+\eps)^{\theta-2} |Df|^2 \zeta^2 \leq \int_\Omega (f+\eps)^{\theta-1} 2\zeta Df \cdot D\zeta .
\]
If we set $w = (f+\eps)^{\theta/2}$ and rearrange then we obtain
\[
\int_\Omega |D(w \zeta)|^2 \leq c(\theta) \int_\Omega w^2 |D\zeta|^2.
\]
This implies $w\zeta \in W^{1,2}(\Omega \cap B_r)$ for all $r < 1$.  If we replace $\zeta$ with $\phi^{\alpha - \beta}$ for $\beta \chi = \beta + 1$ and $\alpha - \beta - 1 > 0$ and $\phi$ as before, then we get
\[
\bigg( \int_\Omega w^{2\chi} \phi^{2\alpha \chi} d\eta \bigg)^{1/\chi} \leq c(\theta, d, \chi) \int_\Omega w^2 \phi^{2\alpha} d\eta
\]
for $d\eta = \phi^{-2\beta-2}dx = \phi^{2\beta \chi} dx$.  Now apply Holder like in \eqref{eqn:john-7} to get, for any $\delta \in (0, 1)$:
\[
\bigg( \int_\Omega w^{2\chi} \phi^{2\alpha \chi} d\eta \bigg)^{(1-\delta)/\chi} \leq c(\theta, d, \chi) \bigg( \int_\Omega (w^2 \phi^{2\alpha})^{(1-\delta)\chi/(\chi-\delta)} d\eta \bigg)^{(\chi-\delta)/\chi}.
\]
Recalling that $\alpha - \beta \chi = \alpha - \beta - 1 > 0$ and our definition of $w$, and taking $\eps \to 0$, we then have
\begin{equation}\label{eqn:john-31}
\bigg( \int_{\Omega \cap B_{\gamma/4}} f^{\theta \chi} \bigg)^{(1-\delta)/\chi} \leq c(\theta, d, \chi) \bigg( \int_{\Omega \cap B_{\gamma/2}} f^{\theta(1-\delta)\chi/(\chi-\delta)} \bigg)^{(\chi-\delta)/\chi}.
\end{equation}
Given any $0 < p < \chi$, we can write $p = \theta \chi$ for $\theta \in (0, 1)$.  We can then choose a $\delta=\delta(p, \chi)$ so that 
\[
\theta (1-\delta) \chi/(\chi - \delta) = \min\{ 1/2e, \theta/2 \}.
\]
Combining \eqref{eqn:john-31}, \eqref{eqn:john-30} with our choice of $\delta$ we obtain
\begin{align*}
&\bigg( \int_{\Omega \cap B_{\gamma/4}} f^p \bigg) \bigg( \int_{\Omega \cap B_{\gamma/4}} f^{-p} \bigg)  \\
&\quad\leq c(p, d, \chi) \bigg( \int_{\Omega \cap B_{\gamma/2}} f^{\min\{1/2e, \theta/2\}} \bigg)^{\frac{\chi-\delta}{1-\delta}} \bigg( \int_{\Omega \cap B_{\gamma/2}} f^{-\min\{ 1/2e, \theta/2\}} \bigg)^{\frac{\chi-\delta}{1-\delta}} \leq c(p, d, \chi),
\end{align*}
which proves \eqref{eqn:john-concl} with $\gamma/4$ in place of $\gamma$.
\end{proof}

\section{Continuity up to the boundary}

In this section, we prove a uniform H\"older estimate for minimizers of the Alt-Caffarelli functional with Lipschitz data on the boundary of a smooth domain, which we use in the proof of \cref{thm:fol}.

\begin{lemma}\label{l:continuity}
	Let $g:\R^{d-1}\to\R$ be a $C^{1,\alpha}$ function and let
	$$\Omega:=\Big\{(x',x_d)\in\R^{d-1}\times\R\ :\ x_d>g(x')\Big\}\qquad\text{and}\qquad \Gamma:=\Big\{(x',g(x'))\ :\ x'\in\R^{d-1}\Big\}.$$
	Let $\varphi:\R^d\to\R$ be a non-negative Lipschitz continuous function and let $u:\Omega\cup \Gamma\to\R$ be a non-negative function in $W^{1,2}_{loc}(\Omega)$ such that $u=\varphi$ on $\Gamma$. Suppose that $u$ satisfies the following minimality condition in a ball $B_R$
	\begin{align*}
	\int_{K}|D u|^2\,dx\le \int_K|D (u+\psi)|^2\,dx+|K|&\quad\text{for every}\quad \psi\in W^{1,2}_0(K)\\
	&\qquad\text{and every open set}\quad K\subset\Omega\cap B_R\,.
	\end{align*}
	Then, $u$ is $\gamma$-H\"older continuous in $B_{\sfrac{R}{2}}\cap(\Omega\cup\Gamma)$ for any $\gamma\in(0,1)$.
\end{lemma}
\begin{proof}
	We define the $C^{1,\alpha}$ map 
	$$\Psi:\Omega\cup\Gamma\to H:=\{(x',y_d)\ :\ y_d\ge 0\}\ ,\qquad \Psi(x',x_d):=(x',x_d-g(x')),$$
	and its inverse 
	$$\Phi:H\to\Omega\cup\Gamma\ ,\qquad \Phi(x',y_d):=(x',y_d+g(x')).$$
	We will prove that the function $u$ satisfies the estimate 
	\begin{equation}\label{e:morrey-estimate}
	\int_{B_r(x_0)}|D u|^2\,dx\le Cr^{d+2(\gamma-1)}
	\end{equation}
	for all $x_0 \in \overline{\Omega} \cap B_{R/2}$, $r < R/4$, and some constant $C>0$ independent of $x_0, r$. Thus, we can apply the Morrey Lemma (see for instance \cite[Lemma 3.12]{Bojo_ln}) to the function $u-\varphi$ obtaining that it is $\gamma$-H\"older continuous, which will conclude the proof. In order to prove \eqref{e:morrey-estimate}, it will suffice to take $x_0 \in \del\Omega$, and for simplicity we can assume that $x_0=0$, $\Phi(0)=0$ and $D\Phi(0)=D\Psi(0)=Id$, and $R = 2$. We also set
	$$A(x):=D\Phi(x)D\Phi(x)^t,$$
	and we notice that there is a constant $C_A$ such that 
	\begin{equation}\label{e:ellipticity-A}
	(1-C_Ar^\alpha)\text{\rm Id}\le A(x)\le (1+C_Ar^\alpha)\text{\rm Id}\quad\text{for every}\quad x\in B_r.
	\end{equation}
	For simplicity, we will denote by $C_d$ any constant depending only on the dimension $d$; by $C_{g}$ we denote  constants depending only on $g,\Phi,\Psi$ and $A$; by $C_\varphi$ we denote constants depending only on $\|\varphi\|_{L^\infty}$ and $\|D\varphi\|_{L^\infty}$.\\

	\noindent{\bf The harmonic extension of $\varphi\circ\Phi$.} Let $h_\varphi:H \cap B_2 \to\R$ be a function such that $\|h_\varphi\|_{L^\infty(H \cap B_2)}\le \|\varphi\|_{L^\infty(H\cap B_2)}$ and 
	$$\Delta h_\varphi=0\quad\text{in}\quad H \cap B_2 \ ,\qquad h_\varphi=\varphi\circ\Phi\quad\text{on}\quad\partial (H \cap B_2).$$
	Given $\eps>0$ and $r>0$, we consider the test function $\widetilde h_\varphi$ solution to 
	$$\Delta \widetilde h_\varphi=0\quad\text{in}\quad H\cap B_{2r^{1-\eps}}\ ,\qquad \widetilde h_\varphi=h_\varphi-\varphi\circ\Phi\quad\text{on}\quad\partial \big(H\cap B_{2r^{1-\eps}}\big).$$
	Then, using the subharmonicity of $|D \widetilde h_\varphi|^2$ and the gradient estimate, we get 
	\begin{align*}
	\int_{H \cap B_r}|D h_\varphi|^2\,dx&\le \int_{H \cap B_r}|D (\varphi\circ\Phi+\widetilde h_\varphi)|^2\,dx\le 2\int_{H \cap B_r}| D (\varphi\circ\Phi)|^2\,dx+2\int_{H \cap B_r}|D \widetilde h_\varphi|^2\,dx\\
	&\le C_dr^d\|D (\varphi\circ\Phi)\|_{L^\infty(H \cap B_r)}^2+C_d \frac{|B_r|}{|B_{r^{1-\eps}}|}\int_{H \cap B_{r^{1-\eps}}}|D \widetilde h_\varphi|^2\,dx\\
	&\le C_{d,\varphi,g} r^d+C_dr^{d\eps}\frac1{r^{2(1-\eps)}}\|\widetilde h_\varphi\|_{L^\infty(H \cap B_{2r^{1-\eps}})}^2\le C_{d,\varphi,g}(r^d+r^{(d+2)\eps-2}).
	\end{align*}
	Now, for any fixed $\beta>0$, we can choose $\displaystyle\eps:=\frac{d+2\beta}{d+2}$, obtaining
	\begin{equation}\label{e:h-varphi-estimate}
	\int_{H \cap B_r}|D h_\varphi|^2\,dx\le C_{d,\varphi,g} r^{d+2(\beta-1)}\quad\text{for every}\quad r\in(0,\sfrac12).
	\end{equation}

	\noindent{\bf Almost-minimality of $u$.} Let $r\in(0,1)$ and let $h$ be the harmonic extension:
	$$\Delta h=0\quad\text{in}\quad H\cap B_r\ ,\qquad h=u\circ\Phi-h_\varphi\quad\text{in}\quad \partial (H\cap B_r)\,,$$
	so in particular, $h\equiv 0$ on $B_r\cap\partial H$. Let $f:=h\circ\Phi^{-1}$. Then 
	%$f$ is $A$-harmonic in the set  $\Omega_r:=\Phi(H\cap B_{r})$, that is
	$$\text{div}(A(x)D f)=0\quad\text{in}\quad \Omega_r\ ,\qquad u=f\quad\text{on}\quad \partial\Omega_r\,,$$
	where $\Omega_r:=\Phi(H\cap B_{r})$. 
	%Now, by the optimality of $u$, we have that 
%	$$\int_{\Omega_r}|D u|^2\,dx\le \int_{\Omega_r}|D f|^2\,dx+|\Omega_r|\,.$$
	Using the equation for $f$, the ellipticity condition \eqref{e:ellipticity-A} and the optimality of $u$ tested with $f$ in the set $\Omega_r$, we get that
	\begin{align*}
	\int_{\Omega_r}D (u-f)\cdot A(x)D (u-f)\,dx
%	&= \int_{\Omega_r}D u\cdot A(x)D u\,dx- \int_{\Omega_r}D f\cdot A(x)D f\,dx-2\int_{\Omega_r}D (u-f)\cdot A(x)D (u-f)\,dx\\
	&= \int_{\Omega_r}D u\cdot A(x)D u\,dx- \int_{\Omega_r}D f\cdot A(x)D f\,dx\\
	&\le(1+C_gr^\alpha)  \left(\int_{\Omega_r}|D u|^2\,dx-\frac{1-C_gr^\alpha}{1+C_gr^\alpha}\int_{\Omega_r}|D f|^2\,dx\right)\\
	&\le(1+C_gr^\alpha)  \left(|\Omega_r|+C_gr^\alpha\int_{\Omega_r}|D f|^2\,dx\right).
%	&\le C_g|\Omega_r|+C_gr^\alpha\int_{\Omega_r}|D f|^2\,dx.
	%&\le C_g|\Omega_r|+C_gr^\alpha\int_{\Omega_r}|D u|^2\,dx.
	\end{align*}
	Using $\displaystyle \int_{\Omega_r}D f\cdot A(x)D f\,dx\le \int_{\Omega_r}D u\cdot A(x)D u\,dx$ and the ellipticity of $A$, we get
	\begin{equation}\label{e:eq-for-difference-u-f}
	\int_{\Omega_r}|D (u-f)|^2\,dx\le C_gr^d+C_gr^\alpha\int_{\Omega_r}|D u|^2\,dx.
	\end{equation}
	\noindent{\bf Main estimate.} We fix a constant $\kappa\in(0,1)$. Using \eqref{e:eq-for-difference-u-f} and \eqref{e:h-varphi-estimate}, we compute
	\begin{align*}
	\int_{\Phi(H\cap B_{\kappa r})}|D u|^2\,dx%&\le2\int_{\Phi(H\cap B_{\kappa r})}|D (u-f)|^2\,dx+2\int_{\Phi(H\cap B_{\kappa r})}|D f|^2\,dx\\
	&\le2\int_{\Phi(H\cap B_{r})}|D (u-f)|^2\,dx+2\int_{\Phi(H\cap B_{\kappa r})}|D f|^2\,dx\\
%	&\le C_g|\Omega_r|+C_gr^\alpha\int_{\Phi(H\cap B_{r})}|D u|^2\,dx+2\int_{\Phi(H\cap B_{\kappa r})}|D f|^2\,dx\\
	&\le C_gr^d+C_gr^\alpha\int_{\Phi(H\cap B_{r})}|D u|^2\,dx+C_g\int_{H\cap B_{\kappa r}}|D h|^2\,dx\\
	&\le C_gr^d+C_gr^\alpha\int_{\Phi(H\cap B_{r})}|D u|^2\,dx\\
	&\qquad+C_g\int_{H\cap B_r}|D h_\varphi|^2\,dx+C_g\int_{H\cap B_{\kappa r}}|D (h+h_\varphi)|^2\,dx\\
	&\le C_{d,\varphi,g}r^{d-2(1-\beta)}+C_gr^\alpha\int_{\Phi(H\cap B_{r})}|D u|^2\,dx+C_g\int_{H\cap B_{\kappa r}}|D (h+h_\varphi)|^2\,dx\,.
	\end{align*}
	Now, since $h+h_\varphi$ is harmonic in $H\cap B_r$ and vanishes on $\partial H\cap B_r$, we obtain
	\begin{align*}
	\int_{\Phi(H\cap B_{\kappa r})}|D u|^2\,dx
	&\le C_{d,\varphi,g}r^{d-2(1-\beta)}+C_gr^\alpha\int_{\Phi(H\cap B_{r})}|D u|^2\,dx+C_g\frac{|B_r|}{|B_{\kappa r}|}\int_{H\cap B_{r}}|D (h+h_\varphi)|^2\,dx\\
	&\le C_{d,\varphi,g}r^{d-2(1-\beta)}+C_gr^\alpha\int_{\Phi(H\cap B_{r})}|D u|^2\,dx+C_g\frac{|B_r|}{|B_{\kappa r}|}\int_{H\cap B_{r}}|D h|^2\,dx\\
	&\le C_{d,\varphi,g}r^{d-2(1-\beta)}+C_g\big(r^\alpha+\kappa^d\big)\int_{\Phi(H\cap B_{r})}|D u|^2\,dx.
	\end{align*}
	\noindent{\bf Iteration estimate and conclusion.} We take $\gamma\in(0,\beta)$ and we set 
	$$r_n=\kappa^n\qquad\text{and}\qquad M_n:=\frac{1}{r_n^{d-2(1-\gamma)}}\int_{\Phi(H\cap B_{r_n})}|D u|^2\,dx.$$
	Then, setting $A:=\kappa^{-2}C_{d,g,\varphi}$ and $b:=2C_{g}\kappa^{2(1-\gamma)}$, we have 
	$$M_{n+1}\le A\kappa^{2n(\beta-\gamma)}+bM_n\quad\text{for every}\quad n\ge \frac{d}{\alpha}.$$
	We now choose $\kappa$ in such a way that $b\le1$. Then, $M_n$ remains bounded by a universal constants. Indeed,  if $n_0$ is the smallest integer greater than $\sfrac{d}{\alpha}$, then 
	$$M_n\le \frac{A}{1-\kappa^{2(\beta-\gamma)}}+M_{n_0}\quad\text{for every}\quad n\ge n_0,$$ 
	which concludes the proof of \eqref{e:morrey-estimate}.
\end{proof}

\end{document}